\documentclass{article}

\usepackage{mathptmx}
\usepackage{amssymb,amsmath,amsthm,amsfonts}
\usepackage{graphicx}
\usepackage{bm}
\usepackage[all,cmtip]{xy}
\usepackage{tikz-cd}
\usepackage{stmaryrd}
\usepackage{xcolor}

\usepackage{MnSymbol}
\definecolor{linkblue}{HTML}{003d73}
\definecolor{linkgreen}{HTML}{006161}
\definecolor{linkred}{HTML}{a11950}
\usepackage[pagebackref]{hyperref}
\hypersetup{
	pdftitle={Full Spark Frames},
	pdfauthor={Tom Needham and Clayton Shonkwiler},
	pdfsubject={frame theory},
	pdfkeywords={frames, full spark, unit-norm tight frames},
	colorlinks=true,
	linkcolor=linkblue,
	citecolor=linkgreen,
	urlcolor=linkred
}
\usepackage{algorithm}
\usepackage{algpseudocode}
\usepackage{pifont}
\usepackage[margin=1in]{geometry}
\usepackage{mathtools}
\usepackage[percent]{overpic}
\usepackage{booktabs}
\usepackage{authblk}
\usepackage{cite}
\usepackage[nameinlink]{cleveref}
\usepackage{quiver}
\usepackage{mathdots}

\usepackage{enumitem}

\crefname{thm}{theorem}{theorems}
\crefname{prop}{proposition}{propositions}

\newtheorem{thm}{Theorem}[section]
\newtheorem*{thm*}{Theorem}
\newtheorem{prop}[thm]{Proposition}
\newtheorem{lem}[thm]{Lemma}
\newtheorem{cor}[thm]{Corollary}

\theoremstyle{definition}
\newtheorem{definition}[thm]{Definition}
\newtheorem{example}[thm]{Example}
\newtheorem{remark}[thm]{Remark}

\newcommand{\R}{\mathbb{R}}

\newcommand{\C}{\mathbb{C}}

\newcommand{\Gr}{\mathrm{Gr}_d(\C^N)}

\newcommand{\Fl}{\operatorname{F\ell}}
\newcommand{\frames}{\mathcal{F}^{d,N}}

\newcommand{\Id}{\mathbb{I}}
\newcommand{\tr}{\operatorname{tr}}

\newcommand{\FUNTF}{\frames_{\left(\frac{N}{d},\dots , \frac{N}{d}\right)}(1,\dots , 1)}

\newcommand{\unitary}{\operatorname{U}(d)}
\newcommand{\hermitian}{\mathcal{H}}
\newcommand{\torus}{\operatorname{U}(1)}
\newcommand{\zu}{Z(\operatorname{U}(N))}
\newcommand{\bt}{\mathbb{T}}
\newcommand{\btt}{\boldsymbol{t}}
\newcommand{\polytope}{\mathcal{P}^{d,N}}
\newcommand{\Ad}{\operatorname{Ad}}
\renewcommand{\Im}{\operatorname{Im}}

\newcommand{\br}{\boldsymbol{r}}
\newcommand{\blam}{\boldsymbol{\lambda}}

\usepackage{color}
\definecolor{darkblue}{rgb}{0.0, 0.0, 0.8}
\definecolor{darkred}{rgb}{0.8, 0.0, 0.0}
\definecolor{darkgreen}{rgb}{0.0, 0.8, 0.0}

\title{Toric Symplectic Geometry and Full Spark Frames}
\author[$\ast$]{Tom Needham}
\author[$\dag$]{Clayton Shonkwiler}
\affil[$\ast$]{Department of Mathematics, Florida State University, Tallahassee, FL} 
\affil[$\dag$]{Department of Mathematics, Colorado State University, Fort Collins, CO}
\date{}

\begin{document}

\maketitle

\begin{abstract}
    The collection of $d \times N$ complex matrices with prescribed column norms and prescribed (nonzero) singular values forms a compact algebraic variety, which we refer to as a \emph{frame space}. Elements of frame spaces---i.e., \emph{frames}---are used to give robust representations of complex-valued signals, so that geometrical and measure-theoretic properties of frame spaces are of interest to the signal processing community. This paper is concerned with the following question: what is the probability that a frame drawn uniformly at random from a given frame space has the property that any subset of $d$ of its columns gives a basis for $\mathbb{C}^d$? We show that the probability is one, generalizing recent work of Cahill, Mixon, and Strawn. To prove this, we first show that frame spaces are related to highly structured objects called toric symplectic manifolds. This relationship elucidates the geometric meaning of \emph{eigensteps}---certain spectral invariants of a  frame---and should be a more broadly applicable tool for studying probabilistic questions about the structure of frame spaces. As another application of our symplectic perspective, we completely characterize the norm and spectral data for which the corresponding frame space has singularities, answering some open questions in the frame theory literature.
\end{abstract}

\section{Introduction}\label{sec:intro}
A \textit{frame} in a Hilbert space $(\mathcal{H},\langle\cdot,\cdot\rangle)$ is traditionally defined as a collection $\{f_i\}_{i \in \mathcal{I}}$ of vectors in $\mathcal{H}$ so that for all $v \in \mathcal{H}$ we have
\begin{equation}\label{eq:frame inequality}
    a \|v\|^2 \leq \sum_{i \in \mathcal{I}} \left| \langle v, f_i \rangle \right|^2 \leq b \|v\|^2
\end{equation}
for some numbers $0 < a \leq b$ called \textit{frame bounds}. In this paper, we focus on finite frames in complex Hilbert spaces, in which case $\mathcal{H} = \C^d$ for some integer $d >0$ with its standard Hermitian inner product and norm, $\mathcal{I} = \{1, \dots , N\}$ is finite, and the above condition is equivalent to the collection $\{f_1, \dots, f_N\}$ being a spanning set for $\C^d$. 

Interest in finite frames is largely due to their application to robust signal representation. Modeling a signal as an element of a Hilbert space $\mathcal{H}$, a frame allows one to take a sequence of ``measurements" by recording the inner product of the signal with each of the frame vectors. This signal representation is more robust to noise in the signal or random erasures of measurements than a similar measurement scheme associated to an orthonormal basis, at least when the frame has certain properties \cite{Goyal:2001cd,Casazza:2003vp,holmes2004optimal}. These desirable properties for a frame are typically expressed as prescriptions for the norms of the frame vectors and for the spectrum of an operator associated to the frame, which we describe in more detail below. The collection of frames with prescribed norm and spectral data is easily seen to define an algebraic variety, and there is interest in the geometric structure of these frame varieties \cite{Dykema:2006ux,Strawn:2010bn,Cahill:2017gv}. This paper explores the geometry of frame varieties through the lens of \emph{symplectic geometry}. Symplectic geometry is a subfield of differential geometry that studies manifolds endowed with a certain geometric structure---to keep the paper accessible to the broader frame theory community, we have aimed to give self-contained expositions of the relevant ideas from symplectic geometry, but we will not go into details here in the introduction. 

The main contributions of this paper are as follows:
\begin{itemize}
    \item Our main theorem, \Cref{thm:main}, says that with probability 1 every size-$d$ subset of a random frame in $\C^d$ with prescribed norms and spectral data is a basis. Frames satisfying this non-degeneracy condition are called \textit{full spark frames}. This generalizes the complex case of Theorem~1.6 from Cahill, Mixon, and Strawn's paper~\cite{Cahill:2017gv}, which established the genericity of the full spark condition for frames whose frame bounds are equal (this can be stated as a spectral constraint) and whose frame vectors are all unit---this is referred to as the space of \emph{finite unit-norm tight frames} or \emph{FUNTFs}.
    \item We show in \Cref{thm:frame space manifold eigenvalue characterization} that the space of frames with prescribed norm and spectral data is a smooth manifold, for generic choices of this data. In fact, we give necessary and sufficient conditions on the norms and eigenvalues which guarantee that the corresponding frame variety is smooth, and we describe the local geometry of singular varieties near their singular points. This theorem (together with surrounding results) generalizes work of Dykema and Strawn~\cite{Dykema:2006ux}, which once again specializes to the space of FUNTFs, and answers generalizations of two open questions posed in~\cite{Cahill:2017gv}.
    \item Both of the previous two results are proved by novel applications of ideas from symplectic geometry to frame theory. Throughout the course of the paper, we show that many spaces of complex frames have natural interpretations from the symplectic point of view. In particular, we show in \Cref{thm:toric_symplectic} that each space of frames with prescribed spectral and norm data has a dense open subset which projects onto a highly structured geometric object called a \emph{toric symplectic manifold}. This geometric structure has measure-theoretic implications---we use it to prove our full spark theorem, but expect that it will be a useful tool for the future study of probability theory on frame spaces. \Cref{thm:toric_symplectic} generalizes work of Flaschka and Millson~\cite{Flaschka:2005dq}, which is written in the context of pure symplectic geometry and makes no references to frames. The theorem also gives a symplectic interpretation of the frame-theoretic concept of \emph{eigensteps}, introduced by Cahill et al.~\cite{Cahill:2013jv}, and an auxiliary result used in its proof generalizes a theorem of Haga and Pegel~\cite{Haga:2016jva}.
\end{itemize}
To describe our results in detail, we now  introduce more precise terminology and notation.

\paragraph{Notation and statement of the main theorem.}

Let $\frames$ be the space of frames of $N$ vectors in $\C^d$. Identifying a frame $\{f_i\}_{i=1}^N \in \frames$ with the $d \times N$ matrix whose columns are the $f_i$ represented in the standard basis, the space $\frames$ can be viewed as an open, dense subset of the space $\C^{d \times N}$ of $d \times N$ complex matrices.

When $a=b$ in~\eqref{eq:frame inequality}, the frame satisfies a scaled Parseval identity
\[
    \sum_{i=1}^N \left|\langle v, f_i\rangle \right|^2 = a \|v\|^2,
\]
and such frames are called $a$-\textit{tight} (or just \textit{tight}); 1-tight frames are called \textit{Parseval frames}.

Each frame $\{f_i\}_{i=1}^N \in \frames$ has three related operators, which can be interpreted in terms of the $d \times N$ matrix $F = \begin{bmatrix} f_1 \mid f_2 \mid \dots \mid f_N\end{bmatrix}$:
\begin{enumerate}
    \item The \textit{analysis operator} $\C^d \to \C^N$ is defined by
    \begin{equation*}
        v \mapsto \left(\langle v, f_1\rangle , \dots , \langle v, f_N \rangle \right),
    \end{equation*}
    or equivalently $v \mapsto F^* v$, where $F^*$ is the Hermitian adjoint (i.e., conjugate transpose) of $F$;
    \item The \textit{synthesis operator} $\C^N \to \C^d$ is defined by
    \begin{equation*}
        (w_1, \dots , w_N) \mapsto \sum_{i=1}^N w_i f_i,
    \end{equation*}
    or equivalently $w \mapsto F w$;
    \item The \textit{frame operator} $\C^d \to \C^d$ is the composition of the analysis and synthesis operators; i.e., $v \mapsto FF^* v$.
\end{enumerate}

A simple calculation shows that a frame is $a$-tight if and only if its frame operator is $a \Id_d$, where $\Id_d$ is the identity map on $\C^d$. 

Frame operators are always Hermitian and positive-definite, so they have spectrum ${\lambda_1 \geq \dots \geq \lambda_d > 0}$, which we will call the \textit{frame spectrum}. If $\blam = (\lambda_1, \dots , \lambda_d)$, we will use $\frames_{\blam}$ to indicate the frames with frame spectrum $\blam$. Notice that the $a$-tight frames are uniquely specified by their frame spectra: $\frames_{(a,\dots , a)}$ is the space of all $a$-tight frames.

In addition to specifying a frame operator, we also often want to fix the (squared) norms of the individual frame vectors. We can always permute the labels on the frame vectors, so it will be convenient in what follows to assume the norms are sorted in decreasing order. If $\br = (r_1, \dots , r_N)$ is a non-increasing list of positive numbers $r_1 \geq  \dots \geq r_N > 0$,\footnote{We could also allow some of the $r_i$ to be zero, but this would complicate some statements below to no apparent benefit.} we will use $\frames(\br)$ to indicate the space of frames with $\|f_i\|^2 = r_i$. It is especially common to require that all the frame vectors have the same norm: $\|f_i\|^2 = r>0$ for all $i$, in which case the frame is an \textit{equal-norm frame}; if $r=1$, this is a \textit{unit-norm frame}. 

In general, the frame norms determine the trace of the frame operator:
\begin{equation}\label{eq:norm trace}
    \sum_{i=1}^N r_i = \sum_{i=1}^N \|f_i\|^2 = \tr F^*F = \tr FF^* = \sum_{i=1}^d \lambda_i,
\end{equation}
by the cyclic invariance of trace. Hence, for tight frames the frame operator must be $\frac{\|\br\|^2}{d}\Id_d$ and a unit-norm tight frame must have $\lambda_i = \frac{N}{d}$ for all $i$. If $\br = (r_1, \dots , r_N)$ and $\blam = (\lambda_1, \dots , \lambda_d)$ with $r_1 \geq \dots \geq r_N > 0$ and $\lambda_1 \geq \dots \geq \lambda_d > 0$, we will use $\frames_{\blam}(\br) := \frames_{\blam} \cap \frames(\br)$ to denote the space of frames $\{f_i\}_{i=1}^N$ with $\|f_i\|^2 = r_i$ and frame spectrum $\blam$. This space has a natural probability measure: the (normalized) Hausdorff measure it inherits as a compact subset of $\C^{d \times N}$.

Equation~\eqref{eq:norm trace} is not the only restriction imposed on the frame vector norms by the frame operator: the partial sums of the squared frame vector norms must be bounded above by the partial sums of the eigenvalues of the frame operator~\cite{Casazza:2010ti,Flaschka:2005dq}. More precisely, there exists a frame $\{f_i\}_{i=1}^N$ with $\|f_i\|^2 = r_i > 0$ and frame spectrum $\blam$ if and only if \eqref{eq:norm trace} holds and, for all $k = 1 , \dots , d$, 
\begin{equation}\label{eq:parseval admissible}
    \sum_{i=1}^k r_i \leq \sum_{i=1}^k \lambda_i.
\end{equation}
Given $\blam$, we will call a non-increasing list $\br$ of positive numbers satisfying~\eqref{eq:norm trace} and~\eqref{eq:parseval admissible} $\blam$-\textit{admissible}, and if all inequalities in~\eqref{eq:parseval admissible} are strict we call the list \textit{strongly $\blam$-admissible}. This terminology comes from Casazza and Leon~\cite{Casazza:2010ti}, who showed that $\frames_{\blam}(\br)$ is non-empty if and only if $\br$ is $\blam$-admissible. Although the terminology does not appear in the work of Casazza and Leon, in general one says that the vector $\blam$ \emph{majorizes} the vector $\br$ in the case that the relations \eqref{eq:parseval admissible} hold \cite{marshall1979inequalities}.

For $\{f_i\}_{i=1}^N \in \frames$, the \emph{spark} of $\{f_i\}_{i=1}^N$ is the size of the smallest linearly dependent subset~\cite{donoho_optimally_2003}. The spark is bounded above by $d+1$, and a frame with spark equal to $d+1$ is called a \emph{full-spark frame}. Equivalently, a frame in $\frames$ is full spark if and only if all of its size-$d$ subsets are bases. Full spark frames are often desirable, for example because they provide unique reconstructions of the largest possible class of sparse signals~\cite{donoho_optimally_2003,Alexeev:2012jk}.

We are now ready to state our main theorem: 

\begin{thm}\label{thm:main}
    Suppose $N > d \geq 1$. Let $\br = (r_1, \dots , r_N)$ and $\blam = (\lambda_1 , \dots , \lambda_d)$ be nonincreasing lists of positive real numbers. There are three mutually exclusive possibilities for the space $\frames_{\blam}(\br)$ of frames $\{f_i\}_{i=1}^N$ with $\|f_i\|^2=r_i$ and frame spectrum $\blam$:
    \begin{enumerate}
        \item If $\br$ is not $\blam$-admissible, then $\frames_{\blam}(\br) = \emptyset$.
        \item If $\br$ is $\blam$-admissible but not strongly $\blam$-admissible, then $\frames_{\blam}(\br)$ is nonempty but consists entirely of frames which are not full spark.
        \item If $\br$ is strongly $\blam$-admissible, then full spark frames have full measure in $\frames_{\blam}(\br)$.
    \end{enumerate}
\end{thm}

In particular, since $\br=(1, \dots , 1)$ is strongly $\blam=(\frac{N}{d}, \dots , \frac{N}{d})$-admissible whenever $N > d$ and since $\mathcal{F}^{d,d}_{(1, \dots , 1)}(1,\dots,1)$ consists of orthonormal bases, which are certainly full spark, we have:

\begin{cor}\label{cor:FUNTF}
    For any $N \geq d \geq 1$, the full-spark frames have full measure inside the space $\FUNTF$ of unit-norm tight frames (FUNTFs) in $\C^d$.
\end{cor}

This result essentially recovers the complex case of \cite[Theorem~1.6]{Cahill:2017gv}, where the authors showed that full spark frames form an open dense subset of $\FUNTF$. Our corollary slightly sharpens this result, in that we are able to refer precisely to the canonical probability measure on $\FUNTF$.

\paragraph{Structure of the paper.}

\Cref{sec:symplectic_structure} begins with an exposition of the relevant ideas and tools from symplectic geometry. These tools are then applied to show that various spaces of frames with prescribed data have symplectic structures. Then, parameters $\blam$ and $\br$ for which the frame variety $\frames_{\blam}(\br)$ is a smooth manifold are characterized. \Cref{sec:toric_structure} is devoted to showing that the space $\frames_{\blam}(\br)$ has a dense open subset which projects onto a toric symplectic manifold (we define and describe the basic properties of such an object in \Cref{sec:toric_symplectic_manifold}). The main theorem is proved in \Cref{sec:proof_of_main_theorem}, as an application of the toric symplectic structure from the previous section. An alternative proof using related tools from algebraic geometry is also sketched. The paper concludes with a brief discussion of future directions in \Cref{sec:discussion}.

\section{Symplectic Structure of Frame Varieties}\label{sec:symplectic_structure}

In this section, we describe the symplectic structures on the frame varieties of interest; i.e., spaces of frames with prescribed frame spectra and/or norms. It has been previously observed in the literature that certain parameter choices may lead to frame spaces with singularities \cite{Dykema:2006ux,Strawn:2010bn,Cahill:2017gv}. We give a precise characterization of this phenomenon and describe the local structure of singular points later in this section.

\subsection{Symplectic Geometry}

We begin with a review of some concepts of symplectic geometry, with a focus on various notions of reducing a symplectic manifold by a Lie group action. We use \cite{mcduff2017introduction} as our main reference for the basics of symplectic manifolds. The ideas we present are standard in the field of symplectic geometry; this subsection is mainly intended as a quick reference for non-experts and to standardize our notation.

A \emph{symplectic form} on a (smooth, real) manifold $M$ is a closed, nondegenerate 2-form $\omega$. For a point $p \in M$ and tangent vectors $X,Y \in T_p M$, we write $\omega_p(X,Y) \in \R$ for the evaluation of $\omega$ on the vectors. The closedness condition means that the exterior derivative $d\omega$ is identically zero and the nondegeneracy condition means that for every nonzero $X \in T_p M$ there exists $Y \in T_p M$ such that $\omega_p(X,Y) \neq 0$. A manifold endowed with a symplectic form is called a \emph{symplectic manifold}, denoted $(M,\omega)$ or simply as $M$ when the symplectic form is understood to be fixed. A simple argument shows that if a manifold $M$ admits a symplectic form, then it must be even dimensional (over the reals). 

\begin{example}[Complex $n$-Space] \label{ex:complex space}
The prototypical example of a symplectic manifold is $n$-dimensional complex space $\C^n$, which is considered as a $2n$-dimensional real manifold via the natural identification $\C^n \approx \R^{2n}$. For any $n$-tuple $p$ of complex numbers, there is a natural isomorphism $T_{p} \C^n \approx \C^n$. Coordinates $(x_1 + \sqrt{-1}y_1, \ldots, x_n + \sqrt{-1}y_n)$ for $\C^n$ correspond to real coordinates $(x_1,\ldots,x_n,y_1,\ldots,y_n)$ in which a symplectic form is given by
\begin{equation}\label{eqn:standard_symplectic}
\omega = dx_1 \wedge dy_1 + \cdots + dx_n \wedge dy_n.
\end{equation}
This is referred to as the \emph{standard symplectic form on $\C^{n}$}. In complex coordinates, \eqref{eqn:standard_symplectic} is expressed concretely for $p \in \C^n$ and \[
Z = (z_1,\ldots,z_n), \, W = (w_1,\ldots,w_n) \in T_{p} \C^n \approx \C^n
\]
by
\[
\omega_{p}(Z,W) = -\Im \left(\overline{w_1} z_1 + \cdots + \overline{w_n} z_n \right) = -\Im \left(W^\ast Z\right) = -\Im \left<Z,W\right>,
\]
where $\Im$ denotes the imaginary part of a complex number and $\left<\cdot,\cdot\right>$ is the standard Hermitian inner product on $\C^n$.

The most important example of a symplectic manifold for our purposes is complex matrix space $\C^{d \times N}$, which is just a reshaped version of the space $\C^{d \cdot N}$ described above. For a matrix $F \in \C^{d \times N}$, we have a natural identification $T_F \C^{d \times N} \approx \C^{d \times N}$. The canonical symplectic form on $\C^{d \times N}$ is defined as
\[
\omega_F(X,Y) = -\Im \tr(Y^\ast X).
\]
This is just a transformation of the canonical symplectic form on $\C^{d \cdot N}$ under the reshaping map.

Since the space of frames $\frames \subset \C^{d\times N}$ is an open submanifold, the standard symplectic structure on $\C^{d \times N}$ restricts to make $\frames$ a symplectic manifold.
\end{example}

In fact, \emph{every} $2n$-dimensional symplectic manifold is locally equivalent to $\C^n$ with the standard symplectic form. Let us now make this statement precise. If $\Psi:N \to M$ is a smooth map from a manifold $N$ to a symplectic manifold $(M,\omega)$, the \emph{pullback form} $\Psi^\ast \omega$ on $N$ is defined by
\[
(\Psi^\ast \omega)_p (X,Y) = \omega_{\Psi(p)}(D\Psi(p)(X),D\Psi(p)(Y)),
\]
for $X,Y \in T_p N$ and where $D\Psi(p):T_p N \to T_{\Psi(p)} M$ denotes the derivative of $\Psi$ at $p$. If $(N,\eta)$ is also a symplectic manifold and $\Psi$ is a diffeomorphism with the property that $\Psi^\ast \omega = \eta$, then we say $\Psi$ is a \emph{symplectomorphism} and that $(N,\eta)$ and $(M,\omega)$ are \emph{symplectomorphic}. A fundamental result of symplectic geometry is Darboux's Theorem~\cite[Theorem 3.2.2]{mcduff2017introduction}: every point in a $2n$-dimensional symplectic manifold has an open neighborhood $U$ such that the symplectic manifold $(U,\omega|_U)$ is symplectomorphic to $\C^n$ with the standard symplectic form.

An important aspect of symplectic geometry is the study of interactions between symplectic structures and certain group actions on their manifolds. Let $G$ be a Lie group with Lie algebra $\mathfrak{g}$ and suppose that $G$ acts on a manifold $M$ endowed with a symplectic form $\omega$. For $p \in M$ and $g \in G$, let $g \cdot p \in M$ denote the action of $g$ on $p$. To each $\xi \in \mathfrak{g}$, one associates an \emph{infinitesimal vector field $Y_\xi$ on $M$} via the formula
\[
\left. Y_\xi \right|_{p} := \left.\frac{d}{d\epsilon}\right|_{\epsilon = 0} \exp(\epsilon \xi) \cdot p,
\]
where $\exp:\mathfrak{g} \to G$ is the exponential map of $G$. A map $\Phi:M \to \mathfrak{g}^\ast$, where $\mathfrak{g}^\ast$ denotes the dual to $\mathfrak{g}$, is called a \emph{momentum map} for the $G$-action if its derivative interacts with the symplectic form as follows. Let $D\Phi(p):T_p M \to T_{\Phi(p)} \mathfrak{g}^\ast$ denote the derivative of $\Phi$ at $p \in M$. Then for each $X \in T_p M$, $D\Phi(p)(X) \in T_{\Phi(p)} \mathfrak{g}^\ast \approx \mathfrak{g}^\ast$, where we use the natural isomorphism coming from the fact that $\mathfrak{g}^\ast$ is a vector space. Then $D\Phi(p)(X):\mathfrak{g} \to \R$, and for each $\xi \in \mathfrak{g}$ we require
\[
D\Phi(p)(X)(\xi) = \omega_p(Y_\xi|_p,X).
\]
We also require that the momentum map $\Phi$ is equivariant, in the following sense. Recall that the \emph{adjoint action} of $G$ on $\mathfrak{g}$ is defined, for each $g \in G$, by the map $\Ad_g:\mathfrak{g} \to \mathfrak{g}$ which is the derivative at the identity of the conjugation map $h\mapsto g h g^{-1}$. The corresponding \emph{coadjoint action} of $G$ on the dual Lie algebra $\mathfrak{g}^\ast$ is defined, for each $g \in G$, by the map $\Ad_g^\ast:\mathfrak{g}^\ast \to \mathfrak{g}^\ast$ given by $\Ad_g^\ast(\chi)(\xi) := \chi(\Ad_{g^{-1}}(\xi))$. When $G$ is a matrix group, both the adjoint and coadjoint actions can be interpreted as conjugation actions. The momentum map is required to be equivariant with respect to the given $G$-action on $M$ and the coadjoint action on $\mathfrak{g}^\ast$. Explicitly, this means that, for each $g \in G$ and each $p \in M$,
\[
	\Ad_g^\ast(\Phi(p))=\Phi(g \cdot p).
\]

If a $G$-action admits a momentum map, then we say the action is \emph{Hamiltonian}. Hamiltonian actions give the appropriate setting for performing a quotient operation in the symplectic category. 

\begin{thm}[Marsden--Weinstein--Meyer Theorem for Regular Values~\cite{marsden1974reduction,Meyer:1973wu}]\label{thm:marsden_weinstein}
Let $(M,\omega)$ be a symplectic manifold with a Hamiltonian $G$-action, let $\Phi:M \to \mathfrak{g}^\ast$ be a momentum map for the action and let $\chi \in \mathfrak{g}^\ast$ be a regular value such that $G$ acts freely on the level set $\Phi^{-1}(\chi)$. Then the manifold 
\[
M \sslash_{\chi} G := \Phi^{-1}(\chi)/G, 
\]
called the \emph{symplectic reduction} or \emph{symplectic quotient} over $\chi$, admits a symplectic structure $\omega_{red}$ which is uniquely characterized by the equation
\[
q^\ast \omega_{red} = \iota^\ast \omega,
\]
where $q:\Phi^{-1}(\chi) \to M\sslash_{\chi} G$ is the quotient map and $\iota:\Phi^{-1}(\chi) \to M$ is the inclusion map.
\end{thm}

See~\cite[Example~III.2.18]{Audin:2004bh} for a construction of the complex projective space $\mathbb{CP}^{n-1}$ as a symplectic reduction $\C^n \sslash_\xi U(1)$, where $U(1)$ acts on $\C^n$ by scalar multiplication.

Notice that, in the statement of \Cref{thm:marsden_weinstein}, we required that $G$ acts freely on the level set $\Phi^{-1}(\chi)$. In particular, $G$ acts on this level set, meaning that for any $p \in M$ with $\Phi(p)=\chi$, it must be the case that $\Phi(g \cdot p) = \chi$ for all $g \in G$. By the equivariance required in the definition of the momentum map $\Phi$, this implies that $\Ad_g^\ast(\chi) = \chi$ for all $g \in G$; in other words, $\chi$ must be a fixed point of the coadjoint action of $G$. Conversely, $G$ will act on the level set over any fixed point of the coadjoint action.

More generally, when $\chi \in \mathfrak{g}^\ast$ is not a fixed point of the coadjoint action, we can still take a symplectic reduction over the \emph{coadjoint orbit} of $\chi$, defined to be the set
\[
\mathcal{O}_\chi := \left\{\mathrm{Ad}_g^\ast(\chi) \mid g \in G\right\}.
\]
In this case the equivariance of the momentum map ensures that $G$ acts on $\Phi^{-1}(\mathcal{O}_\chi)$.

\begin{thm}[Marsden--Weinstein--Meyer Theorem for Coadjoint Orbits~\cite{marsden1974reduction,Meyer:1973wu}]\label{thm:marsden_weinstein_coadjoint_orbit}
Let $M$, $\omega$, $G$ and $\Phi$ be as above and let $\chi \in \mathfrak{g}^\ast$ be a regular value. If the action of $G$ on $\Phi^{-1}(\mathcal{O}_\chi)$ is free, then the manifold 
\[
M \sslash_{\mathcal{O}_\chi} G := \Phi^{-1}(\mathcal{O}_\chi)/G
\]
admits a symplectic structure $\omega_{red}$ which is uniquely characterized by $q^\ast \omega_{red} = \iota^\ast \omega$, where $q$ denotes the projection map $\Phi^{-1}(\mathcal{O}_\chi) \to \Phi^{-1}(\mathcal{O}_\chi)/G$ and $\iota$ denotes the inclusion map $\Phi^{-1}(\mathcal{O}_\chi) \to M$.
\end{thm}

It is straightforward to show that when $\chi \in \mathfrak{g}^\ast$ is a regular value, the action of $G$ on $\Phi^{-1}(\mathcal{O}_\chi)$ is at worst locally free. \Cref{thm:marsden_weinstein_coadjoint_orbit} easily generalizes to this setting, with the only modification being that $M \sslash_{\mathcal{O}_\chi} G$ is a symplectic orbifold rather than a symplectic manifold.

Perhaps surprisingly, if we are more concerned with the ``symplectic'' part of this statement than the ``manifold'' (or ``orbifold'') part, Sjamaar and Lerman~\cite{Sjamaar:1991fe} have shown that $\chi$ being a regular value is not an essential hypothesis in \Cref{thm:marsden_weinstein,thm:marsden_weinstein_coadjoint_orbit}. When $\chi$ is a singular value of $\Phi$ then $\Phi^{-1}(\mathcal{O}_\chi)$ is not a manifold, but it nonetheless admits a $G$ action and the quotient still admits a natural symplectic structure.

\begin{thm}[Sjamaar--Lerman~\cite{Sjamaar:1991fe}]\label{thm:sjamaar_lerman}
	Let $M$, $\omega$, $G$ and $\Phi$ be as above and let $\chi \in \mathfrak{g}^\ast$. The symplectic reduction
	\[
	M \sslash_{\mathcal{O}_\chi} G := \Phi^{-1}(\mathcal{O}_\chi)/G
	\]
	is a symplectic stratified space with symplectic structure $\omega_{red}$ characterized by $q^\ast \omega_{red} = \iota^\ast \omega$, where $\iota$ again denotes the inclusion map $\Phi^{-1}(\mathcal{O}_\chi) \to M$.
\end{thm}

Roughly speaking, $M \sslash_{\mathcal{O}_\chi} G$ being a symplectic stratified space means that it is the disjoint union of symplectic manifolds and that these manifolds fit together nicely. Somewhat remarkably, each connected component contains a unique manifold piece which is open and dense:

\begin{thm}[Sjamaar--Lerman~\cite{Sjamaar:1991fe}]\label{thm:sjamaar_lerman2}
	Let $M$, $\omega$, $G$, $\Phi$, and $\chi$ be as above. Then each connected component of the symplectic reduction $M \sslash_{\mathcal{O}_\chi} G$ has a unique open stratum which is a manifold and is connected and dense.
\end{thm}

Kirwan's proof that level sets of proper momentum maps are connected---this follows from~\cite[Remark~3.1]{Kirwan:1984im}, which handles the case where the domain is compact, together with the discussion surrounding~\cite[Remark~9.1]{Kirwan:1984jt}, which applies to proper moment maps---yields the following immediate corollary.

\begin{cor}\label{cor:proper}
	Let $M$, $\omega$, $G$, $\Phi$, and $\chi$ be as above. If $\Phi$ is proper (for example, if $G$ is compact), then the symplectic reduction $M \sslash_{\mathcal{O}_\chi} G$ has a unique open stratum which is a manifold and is connected and dense.
\end{cor}

Therefore, up to sets of lower dimension, the symplectic reduction even over a singular value of a proper momentum map is a connected symplectic manifold.

\subsection{Frame Spaces as Symplectic Quotients}

Various spaces of frames can be endowed with natural symplectic structures via \Cref{thm:marsden_weinstein_coadjoint_orbit,thm:sjamaar_lerman}. To be precise, the symplectic manifolds of interest are actually quotients of frame spaces by certain symmetry groups. Throughout this subsection, fix dimensions $d$ and $N$ and a frame spectrum ${\blam} = (\lambda_1, \dots , \lambda_d)$ with $\lambda_1 \geq \dots \geq \lambda_d > 0$.

\subsubsection{Frames with Prescribed Frame Spectrum}

The unitary group acts on $\frames_{\blam}$ by left matrix multiplication. We then consider the quotient space $\frames_{\blam}/\unitary$, which has a natural symplectic structure, as we will show below. 

Before stating the result, we observe that the space $\hermitian(d)$ of $d \times d$ Hermitian matrices may be identified with the dual Lie algebra $\mathfrak{u}(d)^\ast$ via the isomorphism $\alpha:\hermitian(d) \to \mathfrak{u}(d)^\ast$ taking $\xi$ to the linear functional $\alpha_\xi:\mathfrak{u}(d) \to \R$, defined by
\begin{equation}\label{eqn:hermitian_to_unitary_Lie_algebra}
\alpha_\xi(\eta) = \frac{\sqrt{-1}}{2}\mathrm{tr}(\eta^\ast \xi) = \left< \frac{\sqrt{-1}}{2} \xi, \eta \right>.
\end{equation}
We begin with a preliminary result about this isomorphism. We define a $\mathrm{U}(d)$-action on $\hermitian(d)$ which we (with a slight abuse of terminology) refer to as the \emph{adjoint action}:
\begin{align*}
        \mathrm{Ad}:\mathrm{U}(d) \times \hermitian(d) &\to \hermitian(d) \\
        (A,\xi) &\mapsto \mathrm{Ad}_A(\xi) := A \xi A^\ast.
    \end{align*}

\begin{lem}\label{lem:equivariance}
    The map $\alpha$ defined in \eqref{eqn:hermitian_to_unitary_Lie_algebra} is equivariant with respect to the adjoint action on $\hermitian(d)$ and the coadjoint action on $\mathfrak{u}(d)^\ast$. That is, 
    \[
    \alpha_{\mathrm{Ad}_A(\xi)} = \mathrm{Ad}_A^\ast \alpha_\xi.
    \]
    It follows that, under the map $\alpha$, the coadjoint orbit of any element of $\mathfrak{u}(d)^\ast$ is identified with the collection of Hermitian matrices with a fixed spectrum.
\end{lem}

\begin{proof}
    The first part is a straightforward computation: for $\eta \in \mathfrak{u}(d)$,
    \begin{align*}
        \alpha_{\mathrm{Ad}_A(\xi)}(\eta)  &= \frac{\sqrt{-1}}{2}\mathrm{tr}(\eta^\ast A \xi A^\ast) = \frac{\sqrt{-1}}{2} \mathrm{tr}((A^\ast \eta A)^\ast \xi) = \alpha_\xi(A^\ast \eta A) = \mathrm{Ad}_A^\ast \alpha_\xi(\eta).
    \end{align*}
    
    Any Hermitian matrix $\xi$ with spectrum $\blam$ can be expressed via its eigendecomposition as $\xi = \mathrm{Ad}_A\left(\mathrm{diag}(\blam)\right)$ for some $A \in \mathrm{U}(d)$, so that the second part of the claim follows by equivariance.
\end{proof}

Based on the second part of the lemma, we use $\mathcal{O}_{\blam} \subset \hermitian(d)$ to denote the set of matrices with fixed spectrum $\blam$. 

% We abuse notation and also use $\mathcal{O}_{\blam}$ to denote the associated coadjoint orbit in $\mathfrak{u}(d)^\ast$. 

We now describe the symplectic structure on the space of frames with prescribed frame spectrum.

\begin{prop}[\cite{NeedhamSGC}]\label{prop:fixed_frame_spectrum_symplectic}
The action of $\mathrm{U}(d)$ on $\C^{d \times N}$ by left multiplication is Hamiltonian with momentum map
\begin{align*}
\Phi_{\unitary}:\C^{d \times N} &\to \hermitian(d) \approx \mathfrak{u}(d)^\ast \\
F &\mapsto - FF^\ast.
\end{align*}
It follows that the space $\frames_{\blam}/\unitary$ is the symplectic quotient
\[
\C^{d \times N} \sslash_{\mathcal{O}_{-\blam}} \unitary.
\]
In particular, it has a natural symplectic structure.
\end{prop}

\begin{proof}[Proof Sketch]
We sketch the proof, since the constructions involved will be useful later. For details, see our previous paper~\cite{NeedhamSGC}. 

The fact that $\Phi_{\unitary}:F \mapsto -F F^\ast$ is a momentum map (where we are specifically fixing the isomorphism \eqref{eqn:hermitian_to_unitary_Lie_algebra}) follows by a computation, as shown in \cite[Proposition 2]{NeedhamSGC}. We remark that our momentum map differs by a sign in the present paper, due to slightly different conventions. The fact that $\frames_{\blam} = \Phi_{\unitary}^{-1}(\mathcal{O}_{-\blam})$ follows by \Cref{lem:equivariance}. Moreover, the negative-definite Hermitian matrices are the regular values of $\Phi_{\unitary}$ in $\hermitian(d)$; we proved this by slightly tedious but essentially straightforward calculation in~\cite{NeedhamSGC}, but it also follows easily from \Cref{prop:regular points of momentum maps} below, since the only unitary matrix which acts trivially on a spanning set of $\C^d$ is the identity matrix. Therefore, the entries in $\blam$ being positive means that \Cref{thm:marsden_weinstein_coadjoint_orbit} applies and
\[
	\C^{d \times N} \sslash_{\mathcal{O}_{-\blam}} \unitary = \Phi_{\unitary}^{-1}(\mathcal{O}_{-\blam})/\unitary = \frames_{\blam}/\unitary
\]
is a symplectic manifold.
\end{proof}

\subsubsection{Identification with a Coadjoint Orbit}

The symplectic structure of $\frames_{\blam}/\unitary$ can alternatively be realized by identifying the frame space with a certain coadjoint orbit in the larger space of Hermitian matrices $\hermitian(N)$. This perspective will be useful in the following subsection. In order to state the result precisely, we introduce some more terminology.

It is well known that any coadjoint orbit has a natural symplectic structure, called the \emph{Kirillov--Kostant--Souriau (KKS) symplectic form}~\cite[II.3.c]{Audin:2004bh}), which we denote generically as $\omega^{\mathrm{KKS}}$. Indeed, for an arbitrary Lie group $G$ with Lie algebra $\mathfrak{g}$, the tangent space to a coadjoint orbit $\mathcal{O}_\beta$ at $\beta \in \mathfrak{g}^\ast$ consists of vectors of the form $\mathrm{ad}_\xi^\ast \beta$, where $\xi \in \mathfrak{g}$, and where $\mathrm{ad}_\xi^\ast$ is the coadjoint representation of $\mathfrak{g}$ on $\mathfrak{g}^\ast$, obtained as the derivative of the map $\mathrm{Ad}^\ast: G \to \mathrm{Aut}(\mathfrak{g}^\ast)$ at the identity. Then the KKS form is defined by
\begin{equation}\label{eq:KKS}
\omega^{\mathrm{KKS}}_\beta(\mathrm{ad}_\xi^\ast \beta, \mathrm{ad}_{\xi'}^\ast \beta) := \beta([\xi,\xi']),
\end{equation}
where $[\cdot,\cdot]$ is the Lie bracket on $\mathfrak{g}$. 

By another slight abuse of terminology, we define the \emph{adjoint representation} of $\mathrm{u}(d)$ on $\hermitian(d)$ by
\begin{align*}
        \mathrm{ad}:\mathfrak{u}(d) \times \hermitian(d) &\to \hermitian(d) \\
        (\eta,\xi) &\mapsto \mathrm{ad}_\eta(\xi) := \eta \xi - \xi \eta.
    \end{align*}
An argument similar to that of \Cref{lem:equivariance} proves the following.

\begin{lem}
    The map $\alpha$ defined in \eqref{eqn:hermitian_to_unitary_Lie_algebra} is equivariant with respect to the adjoint representation of $\mathfrak{u}(d)$ on $\hermitian(d)$ and its coadjoint representation on $\mathfrak{u}(d)^\ast$. That is,
    \[
    \alpha_{\mathrm{ad}_\eta(\xi)} = \mathrm{ad}^\ast_\eta \alpha_\xi.
    \]
\end{lem}
    
The adjoint representation  of $\mathrm{u}(d)$ on $\hermitian(d)$ is obtained by differentiating the adjoint action of $\mathrm{U}(d)$ on $\hermitian(d)$. It follows that the tangent space to $\mathcal{O}_{-\blam} \subset \hermitian(d)$ at $\xi$ consists of vectors of the form $\mathrm{ad}_\eta(\xi)$ for $\eta \in \mathfrak{u}(d)$. 

\begin{lem}\label{lem:pullback}
    The pullback of $\omega^{\mathrm{KKS}}$ to $\mathcal{O}_{-\blam} \subset \hermitian(d)$ via $\alpha$ is given by
    \[
    \left(\alpha^\ast \omega^{\mathrm{KKS}}\right)_\xi (\mathrm{ad}_\eta(\xi), \mathrm{ad}_{\eta'}(\xi)) = \mathrm{Im} \; \mathrm{tr}(\xi \eta \eta').
    \]
\end{lem}

\begin{proof}
    We have
    \begin{align*}
        \left(\alpha^\ast \omega^{\mathrm{KKS}}\right)_\xi (\mathrm{ad}_\eta(\xi), \mathrm{ad}_{\eta'}(\xi)) &= \omega^{\mathrm{KKS}}_{\alpha_\xi} (D\alpha(\xi)(\mathrm{ad}_\eta(\xi)), D\alpha(\xi)(\mathrm{ad}_{\eta'}(\xi))) = \omega^{\mathrm{KKS}}_{\alpha_\xi}(\alpha_{\mathrm{ad}_\eta(\xi)}, \alpha_{\mathrm{ad}_{\eta'}(\xi)}) \\
        &= \omega^{\mathrm{KKS}}_{\alpha_\xi}(\mathrm{ad}^\ast_\eta \alpha_\xi, \mathrm{ad}^\ast_{\eta'} \alpha_\xi) = \alpha_\xi([\eta,\eta']) \\
        &= \frac{\sqrt{-1}}{2} \mathrm{tr} ((\eta \eta' - \eta' \eta)^\ast \xi) = \frac{\sqrt{-1}}{2} \cdot -2\sqrt{-1} \cdot \mathrm{Im} \; \mathrm{tr}(\xi \eta \eta') \\
        &= \mathrm{Im} \; \mathrm{tr}(\xi \eta \eta'),
    \end{align*}
where the second-to-last line uses cyclic invariance and linearity of trace. 
\end{proof}

We can now describe precisely how $\frames_{\blam}/\unitary$ is equivalent to a coadjoint orbit in $\hermitian(N)$. 

\begin{prop}[\cite{NeedhamSGC}]\label{prop:diffeomorphic_to_flag}
    Let $\blam$ be a frame spectrum and define $\widetilde{\blam} := (\lambda_1,\ldots,\lambda_d,0,\ldots,0) \in \R^N$ to be $\blam$ padded with $N-d$ zeros. The space $\frames_{\blam}/\unitary$ is diffeomorphic to $\mathcal{O}_{\widetilde{\blam}} \subset \hermitian(N)$, via the map
    \begin{equation}\label{eqn:gram_matrix_map}
    [F] \mapsto F^\ast F
    \end{equation}
    taking a unitary equivalence class to its Gram matrix. 
    
    Moreover, the map \eqref{eqn:gram_matrix_map} is a symplectomorphism with respect to the reduced symplectic structure on $\frames_{\blam}/\unitary$ and the pullback of the Kirillov--Kostant--Souriau form $\alpha^\ast \omega^{\mathrm{KKS}}$ described in \Cref{lem:pullback}.
\end{prop}

The proof of the proposition makes use of the following lemma.

\begin{lem}\label{lem:tangent_space}
    Let $[F] \in \frames_{\blam}/\unitary$. The tangent space $T_{[F]} \frames_{\blam}/\unitary$ is naturally isomorphic to the vector space
    \[
    \{F\zeta \mid \zeta \in \mathfrak{u}(N)\}/\{\xi F \mid \xi \in \mathfrak{u}(d)\}.
    \]
\end{lem}

\begin{proof}
    Any $F \in \frames_{\blam}$ has singular value decomposition of the form $F = U \Sigma_{\blam} V^\ast$, where
    \[
    \Sigma_{\blam} := \left[
    \mathrm{diag}(\blam)^{\frac{1}{2}} \mid 0 \right].
    \]
    Therefore 
    \[
    \frames_{\blam} = \{U\Sigma_{\blam} V^\ast \mid U \in \mathrm{U}(d) \mbox{ and } V \in \mathrm{U}(N)\}.
    \]
    We claim that
    \[
    T_F \frames_{\blam} = \{\xi F - F \zeta \mid \xi \in \mathfrak{u}(d) \mbox{ and } \zeta \in \mathfrak{u}(N)\}.
    \]
    Indeed, a smooth path $F_t$ in $\frames_{\blam}$ with $F_0 = F = U\Sigma_{\blam}V^\ast$ can be expressed as $F_t = U_t \Sigma_{\blam} V_t^\ast$ for some smooth paths $U_t$ in $\mathrm{U}(d)$ with $U_0 = U$ and $V_t \in \mathrm{U}(N)$ with $V_0 = V$. Then the derivative of $F_t$ at $t= 0$ satisfies
    \[
    \dot{F}_0 = \dot{U}_0 \Sigma_{\blam} V_0^\ast + U_0 \Sigma_{\blam} \dot{V}_0^\ast = \xi U \Sigma_{\blam} V^\ast - U \Sigma_{\blam} V^\ast \zeta = \xi F - F \zeta,
    \]
    for $\xi \in \mathfrak{u}(d)$ and $\zeta \in \mathfrak{u}(N)$ satisfying $\dot{U}_0 = \xi U$ and $\dot{V}_0 = \zeta V$, respectively.
    
    By general principles of quotient manifolds, there is a natural isomorphism
    \[
    T_{[F]} \left(\frames_{\blam}/\unitary\right) \approx T_F \frames_{\blam} / T_F (\mathrm{U(d)}\cdot F).
    \]
    Indeed, the quotient map $\frames_{\blam} \to \frames_{\blam}/\mathrm{U}(d)$ is a submersion, so the rank theorem \cite[Theorem 4.12]{Lee:2013bn} implies that, locally near $F$, it is an orthogonal projection with the fiber $\mathrm{U}(d)\cdot F$ being sent to the origin. This gives $T_F \frames_{\blam} \approx T_F (\mathrm{U(d)}\cdot F) \oplus T_{[F]} \left(\frames_{\blam}/\unitary\right)$, and the claim follows. Since $T_F (\mathrm{U(d)}\cdot F) = \{\xi F \mid \xi \in \mathfrak{u}(d)\}$, the lemma follows.
 \end{proof}

\begin{proof}[Proof of \Cref{prop:diffeomorphic_to_flag}]
    The diffeomorphism result is Proposition~3 in our previous paper \cite{NeedhamSGC}, so we only sketch the details here. Unitary equivalence classes of frames are uniquely determined by their Gram matrices (see, for example,~\cite[\S 3.4]{Waldron:2018wc}), so we can identify $\frames_{\blam}/\unitary$ with the elements of $\hermitian(N)$ which arise as Gram matrices of frames in $\frames_{\blam}$. Of course, the frame operator $FF^\ast$ and the Gram matrix $F^\ast F$ have the same nonzero eigenvalues, so specifying the spectrum of the frame operator also determines the spectrum of the Gram matrix by padding with $N-d$ zeros, and it follows that this subset consists of those $N \times N$ Hermitian matrices with spectrum $\widetilde{\blam}$. In other words, \eqref{eqn:gram_matrix_map} is a diffeomorphism between $\frames_{\blam}/\unitary$ and $\mathcal{O}_{\widetilde{\blam}}$.
    
    It remains to prove the symplectomorphism claim. For the rest of this proof, let $\gamma$ denote the map \eqref{eqn:gram_matrix_map}. By \Cref{lem:tangent_space}, the tangent space to $\frames_{\blam}/\unitary$ at $[F]$ consists of equivalence classes (under the quotient operation in the lemma) $[F\zeta]$, where $\zeta \in \mathfrak{u}(N)$. The derivative of $\gamma$ at $[F]$ in the direction $[F\zeta]$ is given by 
    \[
    D\gamma([F])([F\zeta]) = F^\ast F\zeta + (F\zeta)^\ast F = F^\ast F \zeta -  \zeta F^\ast F = \mathrm{ad}_\zeta(F^\ast F).
    \]
    Then, using \Cref{lem:pullback},
\begin{align*}
    \left(\gamma^\ast \alpha^\ast \omega^{\mathrm{KKS}}\right)_{[F]}(F\zeta,F\zeta') &= \left(\alpha^\ast \omega^{\mathrm{KKS}}\right)_{F^\ast F}(\mathrm{ad}_\zeta(F^\ast F), \mathrm{ad}_{\zeta'}(F^\ast F)) = \mathrm{Im} \; \mathrm{tr} (F^\ast F \zeta \zeta').
\end{align*}
On the other hand, the reduced symplectic form on $\frames_{\blam}/\unitary$ evaluates as 
\[
-\mathrm{Im} \; \mathrm{tr} ((F\zeta')^\ast (F\zeta)) = -\mathrm{Im} \; \mathrm{tr} (F^\ast F \zeta (\zeta')^\ast) = \mathrm{Im} \; \mathrm{tr} (F^\ast F \zeta \zeta'),
\]
where we have used the characterization of the reduced form, cyclic invariance of trace and the fact that $\zeta'$ is skew-Hermitian.
\end{proof}

\begin{remark}\label{remark:flag}
    Since $\mathcal{O}_{\widetilde{\blam}} \subset \hermitian(N)$ is identified with a coadjoint orbit, it is diffeomorphic, like all coadjoint orbits, to a flag manifold~\cite[\S II.1.d]{Audin:2004bh}. Specifically, if $\blam$ consists of $\ell$ distinct eigenvalues with multiplicities $k_1,\dots, k_\ell$ and $d_i=k_1 + \dots + k_i$ for $i=1,\dots, \ell$, then
\[
	\mathcal{O}_{\widetilde{\blam}} \approx \Fl(d_1,\dots,d_\ell,N),
\]
the flag manifold whose elements are nested sequences $V_1 \subset \dots \subset V_\ell \subset V_{\ell+1} = \C^N$ of subspaces with $\dim V_i = d_i$. In particular, when $\blam$ is constant then $\ell=1$ and $k_1=d$, so the flag manifold is $\Fl(d,N)$, whose elements are all subspaces $V_1 \subset \C^N$ with $\dim V_1=d$; that is, the Grassmannian $\Gr$ of $d$-dimensional subspaces of $\C^N$.
\end{remark}

\subsubsection{Frames with Prescribed Frame Spectrum and Norms}
\label{sec:prescribed norms and spectrum}

The torus $\torus^N$ can be realized as the subgroup of diagonal elements of $\operatorname{U}(N)$; that is, as the standard maximal torus $\bt \leq \operatorname{U}(N)$. Thought of in this way, $\torus^N \approx \bt$ acts on $\C^{d \times N}$ by right matrix multiplication and, as we will see in a moment, this action preserves the frame spectrum, and hence restricts to an action on $\frames_{\blam}$.

First, though, it is worthwhile to pause and think about the right action of the full unitary group $\operatorname{U}(N)$ on $\C^{d \times N}$. If we try to define the action of $A \in \operatorname{U}(N)$ on $F \in \C^{d \times N}$ by $A \cdot F = F A$ we quickly run into problems: after all, if $A_1, A_2 \in \operatorname{U}(N)$, then 
\[
	A_1 \cdot (A_2 \cdot F) = A_1 \cdot (F A_2) = F A_2 A_1 \neq F A_1 A_2 = (A_1 A_2) \cdot F,
\]
unless $A_1$ and $A_2$ happen to commute. This is not an issue at the level of $\bt$, which is abelian, but it is preferable to define the $\bt$ action to be consistent with an honest $\operatorname{U}(N)$ action. We get such an action by defining $A \cdot F := F A^\ast$ for $A \in \operatorname{U}(N)$ and $F \in \C^{d \times N}$. In particular, if $D \in \bt$, define $D \cdot F := F D^\ast = F \overline{D}$.

This action preserves the frame spectrum: if $F \in \frames_{\blam}$ and $D \in \bt$, then the frame operator of $D \cdot F = F D^\ast$ is $(FD^\ast)(FD^\ast)^\ast = FD^\ast D F^\ast = FF^\ast$, which is the same as the frame operator of $F$, and hence $D \cdot F \in \frames_{\blam}$ as well.

At the frame level, the action of an element of the torus on a frame performs an independent phase rotation on each frame vector. The right action of $\bt$ commutes with the left action of $\unitary$, so the torus action descends to the quotient $\frames_{\blam}/\unitary$, but it is not an effective action. After all, both $\bt$ and $\unitary$ contain a 1-parameter subgroup of scalar matrices, and the actions of these two subgroups cancel each other: if $\Id_k$ denotes the $k \times k$ identity matrix and $e^{i \theta} \in \torus$, then $e^{i\theta}\Id_d \in\unitary$, $e^{i\theta} \Id_N\in \bt$ and for any $F \in \frames_{\blam}$ we have 
\[
	(e^{i\theta}\Id_N) \cdot ((e^{i \theta}\Id_d) F) = (e^{i \theta}\Id_d) F (e^{-i\theta}\Id_N) = F.
\]

We can put this in a more standard context using our identification of $\frames_{\blam}/\unitary$ with the coadjoint orbit $\mathcal{O}_{\widetilde{\blam}} \subset \hermitian(N) \approx \mathfrak{u}(N)^\ast$, described in \Cref{prop:diffeomorphic_to_flag}. The standard maximal torus of $\operatorname{U}(N)$ certainly acts on any coadjoint orbit of $\operatorname{U}(N)$, and in this case the action is just the conjugation action. The scalar matrices are the center $\zu$ of $\operatorname{U}(N)$, whose coadjoint action is trivial. In general, to get an effective action of the maximal torus of a Lie group on a coadjoint orbit of the group, one needs to take the quotient of the torus by the center of the group; in this case, this is simply
\[
	G := \bt/\zu \approx \torus^N/\torus \approx \torus^{N-1},
\]
which can be identified with the subgroup of diagonal elements of $\operatorname{U}(N)$ whose last entry is 1.

Now we see why it was worth defining the action of $\torus^N$ (and hence $G$) on $\frames_{\blam}$ by $D \cdot F = F D^\ast$, since this corresponds exactly to the standard coadjoint action on $\hermitian(N) \approx \mathfrak{u}(N)^\ast$.

In this setting, it is easy to see that the action of $G$ on $\mathcal{O}_{\widetilde{\blam}} \approx \frames_{\blam}/\unitary$ is Hamiltonian:

\begin{prop}\label{prop:Hamiltonian torus action}
	The action of $G$ on $\mathcal{O}_{\widetilde{\blam}}$ is Hamiltonian, with momentum map $\Phi_G: \mathcal{O}_{\widetilde{\blam}} \to \mathfrak{g}^\ast \approx \R^{N-1}$ recording the first $N-1$ diagonal entries.
\end{prop}

\begin{proof}
	First of all, it is standard (see, e.g.,~\cite[Example~5.3.11]{mcduff2017introduction}) that the coadjoint action of $\operatorname{U}(N)$ on the coadjoint orbit $\mathcal{O}_{\widetilde{\blam}}$ is Hamiltonian with momentum map given by the inclusion $\mathcal{O}_{\widetilde{\blam}} \hookrightarrow \mathfrak{u}(N)^\ast \approx \hermitian(N)$; in fact, the analogous statement holds for arbitrary compact Lie groups acting on coadjoint orbits.
	
	Similarly, it is standard (see, e.g.,~\cite[Proposition~III.1.10]{Audin:2004bh}) that the action of the maximal torus $\bt \leq \operatorname{U}(N)$ on $\mathcal{O}_{\widetilde{\blam}}$ is Hamiltonian with momentum map given by the composition
	\[
		\mathcal{O}_{\widetilde{\blam}} \hookrightarrow \hermitian(N) \approx \mathfrak{u}(N)^\ast \to \mathfrak{t}^\ast,
	\]
	where $\mathfrak{t}^\ast$ is the Lie algebra of $\bt$ and the projection $\mathfrak{u}(N)^\ast \to \mathfrak{t}^\ast$ is the one induced by the inclusion $\bt \hookrightarrow \operatorname{U}(N)$. After identifying $\mathfrak{t}^\ast$ with $\R^N$, this projection is easily seen to be the map which records the diagonal entries of a Hermitian matrix.\footnote{It is worth pointing out that this is precisely the setup for the symplectic proof of the Schur--Horn theorem; see Knutson's excellent paper~\cite{Knutson:2000tg} for more.}
	
	More generally, the same applies for any subgroup (see, e.g.,~\cite[p.~213]{CannasdaSilva:2001cg}), and so the action of $G$ on $\mathcal{O}_{\widetilde{\blam}}$ is Hamiltonian with momentum map given by the composition
	\[
		\mathcal{O}_{\widetilde{\blam}} \hookrightarrow \hermitian(N) \approx \mathfrak{u}(N)^\ast \to \mathfrak{g}^\ast,
	\]
	where the projection $\mathfrak{u}(N)^\ast \to \mathfrak{g}^\ast$ is the one induced by the inclusion $G \hookrightarrow \operatorname{U}(N)$. Under the identification $\mathfrak{g}^\ast \approx \R^{N-1}$, this map just records all the diagonal entries but the last one (which in any case is determined by the others, since the trace of any element of $\mathcal{O}_{\widetilde{\blam}}$ is $\sum_i \lambda_i$).
\end{proof}

\begin{remark}
    It is easy to show that the map \eqref{eqn:gram_matrix_map} from $\frames_{\blam}/\unitary$ to $\hermitian(N)$ is equivariant with respect to the right multiplication action of $\mathrm{U}(N)$, $A \cdot [F] = [F A^\ast]$, and the $\mathrm{Ad}$-action of $\mathrm{U}(N)$ on $\hermitian(N)$. This observation, together with \Cref{prop:diffeomorphic_to_flag} and the first paragraph of the proof of \Cref{prop:Hamiltonian torus action}, shows that the action of $\mathrm{U}(N)$ on $\frames_{\blam}/\unitary$ is Hamiltonian, with momentum map $[F] \to F^\ast F$. 
\end{remark}

If $F \in \frames_{\blam}$, then the $(i,i)$ entry of $F^\ast F \in \mathcal{O}_{\widetilde{\blam}}$ is simply $\langle f_i , f_i \rangle = \|f_i\|^2$, so, reinterpreting in terms of frame data, the momentum map of the Hamiltonian $G$-action on $\frames_{\blam}/\unitary$ is simply $[F] \mapsto (\|f_1\|^2,\dots , \|f_{N-1}\|^2)$. Hence, as a subset of $\frames_{\blam}/\unitary$, the collection $\frames_{\blam}(\br)/\unitary$ of unitary equivalence classes of frames in $\frames_{\blam}(\br)$ is precisely the level set $\Phi_G^{-1}(r_1, \dots , r_{N-1})$.

By the Schur--Horn theorem~\cite{schur1923uber, horn1954doubly}, $\Phi_G^{-1}(r_1, \dots , r_{N-1})$ is non-empty only if the vector $(r_1, \dots , r_N)$ lies in the convex hull of the orbit of $\widetilde{\blam} = (\lambda_1 , \dots , \lambda_d, 0, \dots , 0)$ under the action of the symmetric group $S_N$ which permutes entries. When $\br$ and $\blam$ are non-increasing lists, this condition is easily seen to be equivalent to $\br$ being $\blam$-admissible, so $\frames_{\blam}(\br)$ is non-empty precisely when $\br$ is $\blam$-admissible.

When $\br$ is $\blam$-admissible, there is no guarantee that $(r_1, \dots , r_{N-1}) \in \R^{N-1} \approx \mathfrak{g}^\ast$ is a regular value, but we can still apply \Cref{thm:sjamaar_lerman} to see that the quotient of $\Phi_G^{-1}(r_1, \dots , r_{N-1})$ by $G$ is at worst a symplectic stratified space. We have thus proved the first part of the following proposition:

\begin{prop}\label{prop:full reduction}
When $\br$ is $\blam$-admissible, the (non-empty) space $\frames_{\blam}(\br)/(\unitary \times G)$ is the symplectic quotient
\[
\left(\frames_{\blam} /\unitary\right) \sslash_{\br} G,
\]
which is a symplectic stratified space.\footnote{Here we are slightly abusing notation. We are reducing over the point $(r_1,\dots , r_{N-1}) \in \R^{N-1}$, so we should, strictly speaking, use the notation $\left(\frames_{\blam} /\unitary\right) \sslash_{(r_1,\dots, r_{N-1})} G$. This is notationally cumbersome and, since for elements of $\frames_{\blam}$ the quantity $r_N$ is determined by $(r_1,\dots , r_{N-1})$ anyway, it seems overly pedantic to invent a new shorthand for the truncated vector. Hence, we will use $\br$ to indicate both $(r_1,\dots , r_N)$ and $(r_1, \dots , r_{N-1})$ and trust both ourselves and the reader to keep track of which we mean from context.}

Alternatively, $\frames_{\blam}(\br)/(\unitary \times G)$ can be viewed as the symplectic quotient 
\[
\C^{d \times N} \sslash_{\mathcal{O}_{-\blam} \times \{\br\}} (\unitary \times G).
\]
\end{prop}

\begin{proof}
We proved the first sentence above. For the second sentence, Sjamaar and Lerman~\cite[Theorem~4.1]{Sjamaar:1991fe} showed that even when reducing over singular values, we can perform a reduction by a product group in stages. Hence, with $\Phi:\C^{d \times N} \to \hermitian(d) \times \R^{N-1}$ being the (product) momentum map of the Hamiltonian action of $\unitary \times G$ on $\C^{d \times N}$,
\[
	\C^{d \times N}\sslash_{\mathcal{O}_{-\blam} \times \{\br\}} (\unitary \times G) \approx \left(\C^{d \times N}\sslash_{\mathcal{O}_{-\blam}} \unitary \right)\sslash_{\br} G = \left(\frames_{\blam} /\unitary\right) \sslash_{\br} G.
\]
\end{proof}

\begin{remark}
    \Cref{prop:full reduction} says that $\frames_{\blam}(\br)/(\unitary \times G)$ is, in general, a symplectic stratified space. Depending on the parameters $\blam$ and $\br$, it may actually be a smooth manifold. We characterize the parameters such that $\frames_{\blam}(\br)$ is a smooth manifold below in \Cref{thm:frame space manifold eigenvalue characterization}.
\end{remark}

Notice that the momentum map $\Phi: \C^{d \times N} \to \hermitian(d) \times \R^{N-1}$ is given by
\[
	F \mapsto (-FF^\ast, (\|f_1\|^2, \dots , \|f_{N-1}\|^2)).
\]
In particular, the space $\frames_{\blam}(\br)$ is exactly the level set $\Phi^{-1}(\mathcal{O}_{-\blam} \times \{\br\})$. To simplify slightly, assume $\br$ is strongly $\blam$-admissible. When $\frames_{\blam}(\br)$ is a manifold (conditions for which we determine in \Cref{cor:when frame spaces are manifolds} in the next section), then we know its codimension inside $\C^{d \times N}$ is equal to the codimension of $\mathcal{O}_{-\blam} \times \{\br\}$ inside $\hermitian(d) \times \R^{N-1}$. Even when $\frames_{\blam}(\br)$ is not a manifold, we know from \Cref{thm:sjamaar_lerman2} that its quotient by $\unitary \times G$ contains an open dense subset which is a manifold. This subset is a stratum, and hence its (open, dense) preimage is also a manifold by Theorem 3.5 from Sjamaar and Lerman~\cite{Sjamaar:1991fe}. 

In other words, $\frames_{\blam}(\br)$ contains an open, dense subset which is a manifold and which consists precisely of regular points of the momentum map $\Phi$, so for the purposes of dimension-counting we may as well assume $\frames_{\blam}(\br)$ is a manifold.

To that end,
\[
	\dim(\mathcal{O}_{-\blam} \times \{\br\}) = \dim \mathcal{O}_{-\blam} = d^2-k_1^2-\dots - k_\ell^2,
\]
where, as in \Cref{remark:flag}, the $k_1, \dots , k_\ell$ are the multiplicities of the entries in the frame spectrum $\blam$. Subtracting this from $d^2+N-1=\dim(\hermitian(d) \times \R^{N-1})$ gives the codimension $N-1+k_1^2+\dots + k_\ell^2$ of $\mathcal{O}_{-\blam} \times \{\br\}$ (inside $\hermitian(d) \times \R^{N-1}$) and hence also of $\frames_{\blam}(\br)$ (inside $\C^{d \times N}$). This proves:

\begin{cor}\label{cor:dimension}
For $\br$ which is strongly $\blam$-admissible,
	\[
		\dim \frames_{\blam}(\br) = 2dN-N+1-\sum_{j=1}^\ell k_j^2.
	\]
	Hence,
	\[
		\dim \frames_{\blam}(\br)/(\unitary \times G) = 2N(d-1)+2-d^2-\sum_{j=1}^\ell k_j^2,
	\]
where the $k_j$ are the multiplicities of the entries in the frame spectrum $\blam$.
\end{cor}

\subsection{Manifold Structure of Frame Spaces}

While \Cref{cor:proper} will allow us to apply the symplectic machinery even when $\frames_{\blam}(\br)$ is singular, it is still interesting to understand when this space is a smooth manifold. We completely characterize parameters for which $\frames_{\blam}(\br)$ is a manifold in the following subsection. We then give a detailed description of the local structure of singularities in singular frame spaces.

\subsubsection{Characterizing Smooth Frame Spaces}

The model result on manifold structures of frame spaces is due to Dykema and Strawn~\cite{Dykema:2006ux}, who showed that the singular points in the space $\frames_{\left(\frac{N}{d},\dots,\frac{N}{d}\right)}(1,\dots,1)$ of unit-norm tight frames must be orthodecomposable (we recall the definition of this property below). An immediate corollary is that the space of unit-norm tight frames is a manifold when $N$ and $d$ are relatively prime. More generally, Strawn~\cite{Strawn:2010bn} showed that the singular points in any space of frames with fixed frame vector norms and fixed frame operator (rather than frame spectrum) must be orthodecomposable.

We will prove the converse of these results, showing that the singular points (if any) of every $\frames_{\blam}(\br)$ space are exactly the orthodecomposable frames. First, recall what it means for a frame to be orthodecomposable:

\begin{definition}\label{def:orthodecomposable}
	A frame $\{f_i\}_{i=1}^N \subset \C^d$ is \emph{orthodecomposable} if there exists a partition $P_1, \dots , P_m$ of $\{1,\dots , N\}$ and pairwise orthogonal subspaces $V_1, \dots , V_m \subset \C^d$ with $\C^d = \bigoplus_{i=1}^m V_i$ so that, for all $k =1, \dots , m$,  $\{f_i\}_{i \in P_k}$ is a frame for $V_k$.
\end{definition}

In our setting, we are realizing the space $\frames_{\blam}(\br) = \Phi^{-1}\left(\mathcal{O}_{-\blam} \times \left\{\br\right\}\right)$ as the inverse image of a coadjoint orbit under the momentum map $\Phi$. This frame space will certainly be a manifold if it contains no critical points of $\Phi$, so we now characterize the critical points of $\Phi$.

A standard part of the discussion around \Cref{thm:marsden_weinstein,thm:marsden_weinstein_coadjoint_orbit} (see, e.g.,~\cite[Proof of Proposition~5.4.13]{mcduff2017introduction} or~\cite[\S23.2.1]{CannasdaSilva:2001cg}) is the following characterization of regular points of momentum maps:

\begin{prop}\label{prop:regular points of momentum maps}
	Let $H$ be a Lie group and let $(M,\omega)$ be a symplectic manifold admitting a Hamiltonian $H$-action with momentum map $\Psi:M \to \mathfrak{h}^\ast$. Then $p\in M$ is a regular point of $\Psi$ if and only if the action of $H$ is locally free at $p$.
\end{prop}

Recall that a group action is \emph{locally free} at $p$ if the stabilizer of $p$ is discrete. Therefore, \Cref{prop:regular points of momentum maps} gives a characterization of critical points of a momentum map: $p \in M$ is a critical point if and only if the stabilizer of $p$ is continuous, and in particular contains a nontrivial one-parameter subgroup. This is the key to characterizing the critical points of $\Phi: \C^{d \times N} \to \hermitian(d) \times \R^{N-1}$.

\begin{prop}\label{prop:critical points}
	A frame $F$ is a critical point of the momentum map $\Phi:\C^{d \times N} \to \hermitian(d) \times \R^{N-1}$ of the Hamiltonian action of $\unitary \times G$ if and only if $F$ is orthodecomposable.
\end{prop}

\begin{proof}
Let $F = \begin{bmatrix}f_1\mid \cdots \mid f_N\end{bmatrix} \in \C^{d\times N}$ be a frame and let $\zeta:\R \to G$ be a one-parameter subgroup such that $\zeta(t) \cdot F = F$ for all $t \in \R$. Then $\zeta$ uniquely corresponds to an element of the Lie algebra $\mathfrak{u}(d) \times \mathfrak{g}$ of $\unitary \times G$, and in particular there exist matrices $\xi \in \mathfrak{u}(d)$ and $\theta \in \mathfrak{g}$ such that $\zeta(t) = \exp(t(\xi,\theta)) = (\exp(t\xi),\exp(t\theta))$ (see, e.g.,~\cite[Propositions~20.1 and~20.5]{Lee:2013bn}). Under our concrete realization of $G = \bt/Z(\operatorname{U}(N))$ as the subgroup of diagonal unitary matrices with last entry equal to 1, $\theta$ will be of the form
\[
\theta = \sqrt{-1} \cdot \mathrm{diag}(\theta_1,\ldots,\theta_{N-1},0)
\]
for some $\theta_1,\ldots,\theta_{N-1} \in \R$. Then we have
\[
0 = \left.\frac{d}{dt}\right|_{t=0} \zeta(t) \cdot F = \left.\frac{d}{dt}\right|_{t=0}  \exp(t\xi)  F  \exp(-t\theta) = \xi  F - F  \theta.
\]
At the frame vector level, this yields a system of equations
\begin{equation}\label{eqn:orthodecomposable}
\arraycolsep=1.4pt\def\arraystretch{1.4}
\left\{ 
\begin{array}{rcl}
\xi  f_j &=& \sqrt{-1}\theta_j f_j, \qquad j = 1,\ldots,N-1, \\
\xi  f_N &=& 0. 
\end{array}
\right.
\end{equation}

Each $\sqrt{-1}\theta_j$ is therefore an eigenvalue of the skew-Hermitian matrix $\xi$ with corresponding eigenvector $f_j$. Moreover, any solution of \eqref{eqn:orthodecomposable} yields a one-parameter subgroup in the stabilizer of $F$. By the Spectral Theorem for normal matrices, the eigenspaces for distinct eigenvalues of $\xi$ are orthogonal. 

With the above setup, we are prepared to prove the claim. First, suppose that $F$ is not a regular point of the momentum map. By \Cref{prop:regular points of momentum maps}, it is possible to find a one-parameter subgroup $\zeta:\R \to G$ of the stabilizer of $F$ which is nontrivial. In this case, $\xi$ has at least two distinct eigenvalues and it follows that $F$ is orthodecomposable, with frame vectors partitioned into the orthogonal eigenspaces of $\xi$. 

Conversely, suppose that $F$ is orthodecomposable. Let $P_1,\ldots,P_m$ denote the partition and $V_1,\ldots,V_m$ the subspaces appearing in \Cref{def:orthodecomposable}. Permuting columns of $F$ as necessary, we can assume without loss of generality that $P_1 = \{1,\ldots,k\}$ for some $k < N$. For convenience, we can apply a unitary transformation to $F$ so that $f_1,\ldots,f_k$ span a coordinate plane $V_1$; say, the span of the first $\ell < d$ standard basis vectors. Since  $V_2,\ldots,V_m$ are all orthogonal to $V_1$, $\oplus_{i=2}^m V_i$ is the span of the remaining $d-\ell$ standard basis vectors. Then $F$ has the form

\[
	F = \begin{pmatrix} 
		F_1 & 0 \\
		0 & F_2
		\end{pmatrix},
\]
where the nonzero blocks $F_1$ and $F_2$  have sizes $\ell \times k$ and $(d-\ell)\times (N-k)$, respectively. Let 
\[
\xi := \mathrm{diag}(\underbrace{\sqrt{-1},\cdots,\sqrt{-1}}_{\ell},\underbrace{0,\ldots,0}_{d-\ell})\in \mathfrak{u}(d) \qquad \mbox{and} \qquad \theta := \mathrm{diag}(\underbrace{\sqrt{-1},\ldots,\sqrt{-1}}_{k},\underbrace{0,\ldots,0}_{N-k}) \in \mathfrak{u}(1)^N.
\]
It is easy to check that $\xi,\theta$, and $F$ satisfy the system \eqref{eqn:orthodecomposable} and therefore yield a nontrivial one-parameter subgroup in the stabilizer of $F$. By \Cref{prop:regular points of momentum maps}, $F$ is not a regular point of the momentum map.
\end{proof}

The characterization of singular points in $\frames_{\blam}(\br)$ is now an easy corollary:

\begin{cor}\label{cor:when frame spaces are manifolds}
	The singular points of $\frames_{\blam}(\br)$ are exactly the orthodecomposable elements. Hence, $\frames_{\blam}(\br)$ is a manifold if and only if it contains no orthodecomposable frames.
\end{cor}

\begin{proof}
	From \Cref{prop:critical points} we know that the critical points of the momentum map $\Phi$ are exactly the orthodecomposable frames. It is a nontrivial fact (see, e.g.,~\cite[Theorem~5]{Arms:1981ey} or~\cite[Proposition~2.5]{Sjamaar:1991fe}) that in a neighborhood of a critical point of a momentum map, the level set looks like the product of a quadratic cone and a manifold. Therefore, $\frames_{\blam}(\br) = \Phi^{-1}\left( \mathcal{O}_{-\blam} \times \left\{\br\right\}\right)$ has a quadratic singularity at each orthodecomposable frame. Since the structure of such singularities is itself of interest \cite{Cahill:2017gv}, we describe these cone singularities in detail in \Cref{sec:local_structure} below.
	
	Orthodecomposable frames are the only possible singularities: if $F \in \frames_{\blam}(\br)$ is not orthodecomposable, then the differential $D\Phi_{F}:T_{F}\C^{d \times N} \to T_{\Phi(F)}(\mathcal{H}(d) \times \R^{N-1})$ is surjective. In particular, this implies $\Phi$ is transverse to a neighborhood of $\Phi(F)$ in $\mathcal{O}_{-\blam} \times \left\{\br\right\}$, so it is standard that its inverse image---which is a neighborhood of $F$ in $\frames_{\blam}(\br)$---is a submanifold of $\C^{d \times N}$ (see, e.g.,~\cite[Theorem~6.30]{Lee:2013bn}), and in particular is smooth at $F$.
\end{proof}

\begin{remark}
Since \emph{every} orthodecomposable frame is a singular point of the $\frames_{\blam}(\br)$ space containing it, \Cref{cor:when frame spaces are manifolds} solves the complex case of Cahill, Mixon, and Strawn's Problem~5.8~\cite{Cahill:2017gv}.
\end{remark}

To make \Cref{cor:when frame spaces are manifolds} more practical, it would be helpful to know, in terms of some computable conditions on the parameters $d$, $N$, $\blam$, and $\br$, exactly when $\frames_{\blam}(\br)$ contains an orthodecomposable frame.

To that end, suppose $F\in \frames_{\blam}(\br)$ is orthodecomposable. Then, as in the proof of \Cref{prop:regular points of momentum maps}, after possibly permuting columns and applying a unitary transformation, we can realize $F$ in block-diagonal form:
\[
	F = \begin{pmatrix} 
		F_1 & 0 \\
		0 & F_2
		\end{pmatrix},
\]
where $F_1 \in \C^{\ell \times k}$ and $F_2 \in \C^{(d-\ell) \times (N-k)}$ are frames. But then both the Gram matrix and the frame operator associated to $F$ are also block-diagonal:
\[
	F^\ast F = \begin{pmatrix} 
		F_1^\ast F_1 & 0 \\
		0 & F_2^\ast F_2
		\end{pmatrix}
		\quad
		\text{and}
		\quad
		FF^\ast = \begin{pmatrix} 
		F_1 F_1^\ast & 0 \\
		0 & F_2 F_2^\ast
		\end{pmatrix}.
\]
 Since $\tr(F_1^\ast F_1) = \tr(F_1 F_1^\ast)$, we see that $r_1+\dots+r_k=\lambda_1 + \dots + \lambda_\ell$. More generally, since we permuted columns to get $F$ into block-diagonal form, there exist proper subsets $\{r_{i_1}, \dots , r_{i_k}\} \subsetneq \br$ and $\{\lambda_{i_1}, \dots , \lambda_{i_\ell}\}\subsetneq\blam$ so that
\[
	r_{i_1}+ \dots + r_{i_k} = \lambda_{i_1} + \dots + \lambda_{i_\ell}.
\]

Moreover, let 
\[
	\br' = (r_1',\ldots,r_k') := (r_{i_1},\ldots,r_{i_k}) \quad \mbox{and} \quad \blam' = (\lambda_1',\ldots,\lambda_\ell') := (\lambda_{i_1},\ldots,\lambda_{i_\ell})
\]
and let $\br''$ and $\blam''$ be their respective complement vectors. Then we see that the columns $\{f_{i_1},\dots , f_{i_k}\}$ give a frame for an $\ell$-dimensional subspace of $\C^d$ with spectrum $\lambda_{i_1}\geq \dots \geq \lambda_{i_\ell} > 0$, which means that the frame space $\mathcal{F}_{\blam'}^{\ell,k}(\br')$ is non-empty, and in particular that $\br'$ is $\blam'$-admissible. Similarly, the complementary columns give an element of $\mathcal{F}_{\blam''}^{d-\ell,N-k}(\br'')$, and hence $\br''$ is $\blam''$-admissible.

Summarizing, we have shown the following consequences of orthodecomposability:
\begin{lem}\label{lem:orthodecomposable consequence}
    If $F \in \frames_{\blam}(\br)$ is orthodecomposable, then there exist proper partitions $\br = \br' \sqcup \br''$ and $\blam = \blam' \sqcup \blam''$ so that $\br'$ is $\blam'$-admissible, $\br''$ is $\blam''$-admissible, and $\br'$ and $\blam'$ have the same sum (which implies that $\br''$ and $\blam''$ have the same sum as well).
\end{lem}

Combined with \Cref{cor:when frame spaces are manifolds}, this gives a sufficient condition for $\frames_{\blam}(\br)$ to be a manifold. In fact, it is also a necessary condition:

\begin{thm}\label{thm:frame space manifold eigenvalue characterization}
	$\frames_{\blam}(\br)$ is a manifold if and only if there are no proper partitions $\br = \br' \sqcup \br''$ and $\blam = \blam' \sqcup \blam''$ so that $\br'$ is $\blam'$-admissible, $\br''$ is $\blam''$-admissible, and $\br'$ and $\blam'$ have the same sum.
\end{thm}

\begin{proof}
	If $\frames_{\blam}(\br)$ is not a manifold, then \Cref{cor:when frame spaces are manifolds} tells us that it contains some orthodecomposable frame and hence, by \Cref{lem:orthodecomposable consequence}, there are such partitions of $\br$ and $\blam$.
	
	Conversely, suppose there are such partitions. Then, since $\br'$ is $\blam'$-admissible, there exists a frame $\{f_{i_1},\dots , f_{i_k}\} \in \mathcal{F}_{\blam'}^{\ell,k}(\br')$, and similarly there exists a frame $\{g_{j_1},\dots , g_{j_{N-k}}\} \in \mathcal{F}_{\blam''}^{d-\ell, N-k}(\br'')$. Embedding the former into the $\ell$-dimensional subspace of $\C^d$ spanned by the first $\ell$ standard basis vectors and the latter into the $(d-\ell)$-dimensional orthogonal complement and combining them into a single list sorted by decreasing norm gives an orthodecomposable element---and hence a singular point---of $\frames_{\blam}(\br)$, which is therefore not a manifold.
\end{proof}

Notice that \Cref{thm:frame space manifold eigenvalue characterization} recovers Dykema and Strawn's result that the space of unit-norm tight frames is a manifold when $N$ and $d$ are relatively prime, and also implies that $\frames_{\blam}(\br)$ is a manifold for generic $\br$ and $\blam$.

\subsubsection{Describing Singularities}\label{sec:local_structure}

As was mentioned above, \Cref{cor:when frame spaces are manifolds} solves Problem~5.8 of \cite{Cahill:2017gv} for complex FUNTFs. Problem~1.5 of the same paper asks for a description of the local geometry of the space of FUNTFs near orthodecomposable frames, and we now expand on the proof of \Cref{cor:when frame spaces are manifolds} to solve this problem. We describe singularities for progressively more general classes of frames, with full details given in the simplest case and sketches given for more general cases.
	
\paragraph{Minimally Orthodecomposable FUNTFs.} A result of Arms, Marsden, and Moncrief gives a description of the local geometry near any point on the level set of a momentum map of an arbitrary symplectic manifold \cite[Theorem~5]{Arms:1981ey}. This specializes to our setting to show that near any FUNTF 
\[
F \in \frames_{(\frac{N}{d},\ldots,\frac{N}{d})}(1,\ldots,1) = \Phi^{-1}\left(\left(-\frac{N}{d},\ldots,-\frac{N}{d}\right),(1,\ldots,1)\right),
\]
there is a local diffeomorphism of the ambient space $\C^{d \times N}$ which takes a neighborhood of $F$ in the level set onto the product of a quadratic cone $C_F$ and a smooth manifold. The cone is described explicitly as
	\[
	C_F = \left\{X \in \mathrm{ker}(D\Phi(F)) \cap \mathrm{ker}(D\Phi(F) \circ \mathbb{J}) \mid D^2 \Phi(F)(X,X) \in \mathrm{image}(D\Phi(F))\right\},
	\]
where $\mathbb{J}$ denotes multiplication by $\sqrt{-1}$, considered as a linear map on the real vector space $T_F\mathbb{C}^{d \times N}$, and $D^2 \Phi(F)$ is the Hessian of $\Phi$ at $F$, considered as a $\hermitian(d)$-valued bilinear form. Observe that if $F$ is a regular point of $\Phi$ then $C_F$ is just a linear subspace of $\C^{d \times N}$. We will show that $C_F$ is a singular cone when $F$ is not a regular point.
	
Writing $F = \begin{bmatrix}f_1 \mid \cdots \mid f_N\end{bmatrix}$ and $X = \begin{bmatrix}x_1\mid \cdots \mid x_N\end{bmatrix}$, we have 
	\[
	D\Phi(F)(X) = \left(-FX^\ast - XF^\ast,(2\mathrm{Re}\langle f_j, x_j \rangle)_{j=1}^{N-1}\right)
	\]
	and
	\[
	D\Phi(F) \circ \mathbb{J}(X) = \left(\sqrt{-1} FX^\ast - \sqrt{-1} X F^\ast, (2\mathrm{Im}\langle f_j, x_j\rangle)_{j=1}^{N-1} \right),
	\]
	so that 
	\[
	\mathrm{ker}(D\Phi(F)) \cap \mathrm{ker}(D\Phi(F) \circ \mathbb{J}) = \{X \in \C^{d \times N} \mid FX^\ast = 0 \mbox{ and } \langle f_j, x_j \rangle = 0 \; \forall \; j = 1,\ldots,N-1\}.
	\]
	We can also show that
	\[
	D^2 \Phi(F)(X,X) = \left(-2 X X^\ast, (2\|x_j\|^2)_{j=1}^{N-1}\right).
	\]
	
Now suppose that $F$ is orthodecomposable. As in the proof of \Cref{prop:critical points}, we can assume without loss of generality that 
	\[
	F = \begin{pmatrix} 
		F_1 & 0 \\
		0 & F_2
		\end{pmatrix},
    \]
for some submatrices $F_1 \in \C^{d_1 \times N_1}$ and $F_2 \in \C^{d_2 \times N_2}$ (we use a different indexing convention for dimensions than in the proof of \Cref{prop:critical points} for ease of generalization later on). Furthermore, suppose that the blocks $F_1$ and $F_2$ are not themselves orthodecomposable. We then claim that 
    \begin{equation}\label{eqn:spanning_vector}
    \mathrm{image}(D\Phi(F))^\perp = \mathrm{span}\bigg\{ 
    \biggl(\begin{pmatrix} 
		\Id_{d_1} & 0 \\
		0 & 0
		\end{pmatrix}, (\underbrace{1,\ldots,1}_{N_1},\underbrace{0,\ldots,0}_{N_2-1}) \biggr)
    \bigg\} \subset \hermitian(d) \times \R^{N-1},
    \end{equation}
    where the orthogonal complement is taken with respect to the inner product $\llangle \cdot ,\cdot \rrangle$ on $\mathcal{H}^{d\times d} \times \R^{N-1}$ defined by
    \[
    \llangle (\zeta_1,\xi_1),(\zeta_2,\xi_2)\rrangle := \mathrm{Re} \;  \mathrm{trace}(\zeta_2^\ast \zeta_1) + \xi_2^* \xi_1.
    \]
Indeed, a straightforward computation shows that the vector given in \eqref{eqn:spanning_vector} belongs to $\mathrm{image}(D\Phi(F))^\perp$, and we see that this space is one-dimensional as follows. By standard arguments, it suffices to show that the nullity of the dual map $D\Phi(F)^\ast$ is equal to one and Equation 5.2.7 of \cite{mcduff2017introduction} implies that it suffices to show that the map taking an element of $\mathcal{H}(d) \times \R^{N-1}$ to its infinitesimal vector field evaluated at $F$ has one dimensional kernel. By arguments similar to those used in the proof of \Cref{prop:critical points}, this is equivalent to showing that the space of solutions to the system of equations \eqref{eqn:orthodecomposable} is one-dimensional, which holds under our assumption that $F_1$ and $F_2$ are not themselves orthodecomposable (using the fact that the eigenspaces of a skew-Hermitian matrix are orthogonal).

Based on the work in the previous two paragraphs, the cone $C_F$ is described more explicitly as the set of $X \in \C^{d \times N}$ satisfying 
\begin{equation}\label{eqn:subspace_conditions}
FX^\ast = 0, \qquad \qquad \langle f_j,x_j\rangle = 0 \quad \forall \; j=1,\ldots,N-1,
\end{equation}
and 
\begin{equation}\label{eqn:cone_equation}
        0 = \left\llangle \biggl(-2 X X^\ast, (2\|x_j\|^2)_{j=1}^{N-1}\biggr),     \biggl(\begin{pmatrix} 
		\Id_{d_1} & 0 \\
		0 & 0
		\end{pmatrix}, (1,\ldots,1,0,\ldots,0) \biggr) \right\rrangle = 2\sum_{j=1}^{N_2} \|x^{12}_j\|^2 - 2\sum_{j=1}^{N_1} \|x^{21}_j\|^2,
\end{equation}
where we express $X$ as a block matrix with blocks $X^{km}$, $k,m \in \{1,2\}$, of size $d_k \times N_m$ whose column vectors are denoted $x_j^{km}$. 

We can verify that the intersection of the subspace defined by \eqref{eqn:subspace_conditions} and the solution set of the indefinite quadratic equation \eqref{eqn:cone_equation} yields a singular cone by showing that it is not a linear subspace. To do so, choose 
	\[
	X = \begin{pmatrix} 
		0 & X_2 \\
		X_1 & 0
		\end{pmatrix}
    \]
such that $F_1X_1^\ast = 0$ and $F_2 X_2^\ast = 0$, so that $FX^\ast = 0$. The condition $\left<f_j,x_j\right>$ holds automatically due to the block structures of $F$ and $X$, hence \eqref{eqn:subspace_conditions} is satisfied. The matrices $X_j$ can be chosen to be nonzero, since the dimension of the row space of $F_j$ is at most $d_j < N_j$. We can therefore scale the $X_j$'s to have equal Frobenius norm, which means that \eqref{eqn:cone_equation} is also satisfied, so that $X \in C_F$. Next, we observe that 
	\[
	X' = \begin{pmatrix} 
		0 & X_2 \\
		-X_1 & 0
		\end{pmatrix}
    \]
also lies in $C_F$. However, the matrix $X-X'$ does not satisfy \eqref{eqn:cone_equation}, and this proves our claim that $C_F$ is not a linear subspace.

\begin{example}
    Here we provide a concrete example of the structure described above. Let $F$ denote the orthodecomposable unit norm tight frame
    \[
    F = \left(\begin{array}{cccc}
    1 & 1 & 0 & 0 \\
    0 & 0 & 1 & 1 \end{array}\right) \in \mathcal{F}_{(2,2)}^{2,4}(1,1,1,1) =: \mathcal{F}.
    \]
    The cone $C_F$ can be described explicitly as the set of $2 \times 4$ matrices $X$ satisfying equations \eqref{eqn:subspace_conditions} and \eqref{eqn:cone_equation}. Using the $x_j^{km}$ notation from above, we write 
    \[
    X = \left(\begin{array}{cccc}
    x_1^{11} & x_2^{11} & x_1^{12} & x_2^{12} \\
    x_1^{21} & x_2^{21} & x_1^{22} & x_2^{22}
    \end{array}\right).
    \]
    The subspace equations \eqref{eqn:subspace_conditions} then tell us that
    \[
    \left(\begin{array}{cc}
    0 & 0 \\
    0 & 0 \end{array}\right)
    = FX^\ast =
    \left(\begin{array}{cc}
    \overline{x}_1^{11} + \overline{x}_2^{11} & \overline{x}_1^{21} + \overline{x}_2^{21} \\
    \overline{x}_1^{12} + \overline{x}_2^{12} & \overline{x}_1^{22} + \overline{x}_2^{22}
    \end{array}\right)
    \qquad \Leftrightarrow \qquad x_{1}^{km} = - x_2^{km} \quad \forall \; k,m
    \]
    and, letting $f_j$ and $x_j$ denote the columns of $F$ and $X$, respectively, 
    \[
     0 = \langle f_j, x_j \rangle \quad \forall \; j \in \{1,2,3,4\} \qquad \Leftrightarrow \qquad x_1^{11} = x_2^{11} = x_1^{22} = x_2^{22} = 0.
    \]
    The cone equation \eqref{eqn:cone_equation} then simply reads
    \[
    2 \lvert x_1^{12} \rvert^2 = 2 \lvert x_1^{21} \rvert^2 \qquad \Leftrightarrow \qquad  \lvert x_1^{12} \rvert =  \lvert x_1^{21} \rvert.
    \]
    Putting all of this together, the cone $C_F$ is 
    \[
    C_F = \left\{ A(\lambda,\theta,\phi) \mid \lambda, \theta, \phi \in \R \right\}, \qquad \mbox{where} \qquad 
    A(\lambda,\theta,\phi):= \lambda \left(\begin{array}{cccc}
    0 & 0 & e^{\sqrt{-1} \theta} &  - e^{\sqrt{-1} \theta} \\
    e^{\sqrt{-1} \phi} &  - e^{\sqrt{-1} \phi} & 0 & 0 \end{array}\right).
    \]
    Interpreted as a subset of $\C^2 = \{(\lambda_1 e^{\sqrt{-1}\theta},\lambda_2 e^{\sqrt{-1}\phi})\}$, this is precisely the cone over the Clifford torus.
    
    Now consider the map defined by 
    \begin{equation}\label{eqn:cone_map}
    A(\lambda,\theta,\phi) \mapsto \sqrt{1-\lambda^2} F + A(\lambda,\theta,\phi) = 
    \left(\begin{array}{cccc}
    \sqrt{1-\lambda^2} & \sqrt{1-\lambda^2} & \lambda e^{\sqrt{-1} \theta} &  - \lambda e^{\sqrt{-1} \theta } \\
    \lambda e^{\sqrt{-1} \phi} &  - \lambda e^{\sqrt{-1} \phi} & \sqrt{1-\lambda^2} & \sqrt{1-\lambda^2}
    \end{array}\right),
    \end{equation}
    with domain $\{A(\lambda,\theta,\phi) \mid \lambda \in [0,1)\}$, a neighborhood of the origin in $C_F$. We claim that the image of this map is an open subset of $\mathcal{F}$ intersected with a \emph{slice} of the $\left(\mathrm{U}(2) \times \mathrm{U}(1)^4)\right)/\mathrm{U}(1)$-action on $\C^{2 \times 4}$ through $F$---that is, a subset $\mathcal{S} \subset \C^{2 \times 4}$ containing $F$ such that the orbit of any point in $\C^{2 \times 4}$ sufficiently close to $F$ has a unique representative in $\mathcal{S}$.\footnote{Strictly speaking, we should consider the group modulo the isotropy group of $F$ when constructing the slice, but we are suppressing this technical detail for the sake of brevity.}
    
    Specifically, the slice $\mathcal{S}$ is the set of matrices of the form
    \[
    \left(\begin{array}{cccc}
    \alpha_1 & \alpha_2 & \beta_1 e^{\sqrt{-1}\theta} & \beta_2 e^{\sqrt{-1}\theta} \\
    \gamma_1 e^{\sqrt{-1}\phi} & \gamma_2 e^{\sqrt{-1}\phi} & \delta_1 & \delta_2
    \end{array}\right),
    \]
    where $\alpha_j$ and $\delta_j$ are positive real numbers, and $\beta_j$, $\gamma_j$, $\theta$, $\phi$ are arbitrary real numbers. The proof that $\mathcal{S}$ is a slice is somewhat technical, so we omit it. 
    
    It is then straightforward to show that $\mathcal{F} \cap \mathcal{S}$ is equal to the image of our map (i.e., points of the form \eqref{eqn:cone_map}), which implies that a neighborhood of $F$ in $\mathcal{F}$ is identified with the product of a neighborhood of the origin in $C_F$ and the group. 
\end{example}

\paragraph{General Orthodecomposable FUNTFs.} Now suppose that $F$ is a general orthodecomposable FUNTF. After applying isometries as necessary, we can assume without loss of generality that $F$ is a block diagonal matrix $F = \mathrm{diag}(F_1,\ldots,F_\ell)$, where $F_j$ a non-orthodecomposable matrix of size $d_j \times N_j$. Similar computations show that the cone $C_F$ consists of matrices $X \in \C^{d \times N}$ satisfying the linear conditions \eqref{eqn:subspace_conditions} and the system of quadratic equations
\begin{equation*}
    \sum_{m \neq k} \left(\sum_{j=1}^{N_m} \|x_j^{km}\|^2 - \sum_{j=1}^{N_k} \|x_j^{mk}\|^2 \right) = 0 \qquad \forall \; k = 1,\ldots, \ell,
\end{equation*}
where we express $X$ as a block matrix with blocks $X^{km}$ of size $d_k \times N_m$ and with column vectors denoted $x_j^{km}$. Arguments similar to the above show that $C_F$ is a singular cone.

\paragraph{Orthodecomposable Frames in Arbitrary Frame Spaces.} These descriptions of the local structure of FUNTF space near orthodecomposable frames do not intrinsically use the unit norm condition---they only rely on the assumption that the columns have some fixed collection of norms. The arguments therefore apply to describe singularities of spaces of tight frames with fixed norms.

In fact, similar local characterizations can be derived near orthodecomposable frames in any frame space $\frames_{\blam}(\br)$ with an added technical step. The result of Arms, Marsden, and Moncrief specifically treats zero level sets of momentum maps, and extends trivially to handle level sets of fixed points of the $\mathrm{Ad}^\ast$-action on the dual of the Lie algebra. The general frame space $\frames_{\blam}(\br)$ is a level set of a coadjoint orbit: $\frames_{\blam}(\br) = \Phi^{-1}(\mathcal{O}_{-\blam} \times \{\br\})$. The standard \emph{shifting trick} of symplectic geometry (see, e.g., \cite[p. 376]{Sjamaar:1991fe}) can be used to realize $\frames_{\blam}(\br)$ as the 0-level set of an associated momentum map on the symplectic manifold $\C^{d \times N} \times \mathcal{O}_{-\blam}$; more specifically, this space is endowed with a product form, where the symplectic form on the second factor is $-\omega^{\mathrm{KKS}}$ with $\omega^{\mathrm{KKS}}$ being the canonical Kirillov--Kostant--Souriau symplectic form on a coadjoint orbit defined in~\eqref{eq:KKS}. After applying the shifting trick, similar computations can be done to describe the local geometry of singular points.

\section{Toric Geometry of Frame Space}\label{sec:toric_structure}

\subsection{Toric Symplectic Manifolds}\label{sec:toric_symplectic_manifold}

Recall that our goal is to show that the collection of full-spark frames in $\frames_{\blam}(\br)$ has full measure. As mentioned in the introduction, the natural measure on $\frames_{\blam}(\br)$ is the Hausdorff measure it inherits as a compact subset of $\C^{d \times N}$. This measure is not so easy to get a handle on directly, but symplectic geometry provides a more tractable approach. 

All symplectic manifolds are equipped with natural measures. Specifically, suppose $(M,\omega)$ is a symplectic manifold of dimension $2n$. By the non-degeneracy of the symplectic form $\omega$, the maximal wedge power 
\[
\omega^{\wedge n} = \underbrace{\omega \wedge \dots \wedge \omega}_n
\]
is nowhere-vanishing, and hence defines a volume form and associated \emph{symplectic} or \emph{Liouville measure} $m_{\omega}$: for a Borel set $U \subset M$, $m_{\omega}(U) := \int_U \omega^{\wedge n}$.

While $\frames_{\blam}(\br)$ is not symplectic, by \Cref{prop:full reduction} its quotient 
\[
\frames_{\blam}(\br)/(\unitary \times G) \simeq \C^{d \times N} \sslash_{\mathcal{O}_{-\blam} \times \{\br\}} (\unitary \times G)
\]
is. Moreover, since $\C^{d \times N}$ is K\"ahler---i.e., it is a complex manifold with a Hermitian metric whose negative imaginary part defines the standard symplectic form and whose real part defines the standard Riemannian metric---so is the symplectic reduction, and the symplectic measure agrees with the pushforward of Hausdorff measure on $\frames_{\blam}(\br)$ by the quotient map $\frames_{\blam}(\br) \to \frames_{\blam}(\br)/(\unitary \times G)$~\cite[Theorem~3.1]{hitchin_hyperkahler_1987}.

Therefore, to prove our main theorem it suffices to prove that the equivalence classes of full-spark frames have full measure in $\frames_{\blam}(\br)/(\unitary \times G)$, which we can do symplectically. The key is that $\frames_{\blam}(\br)/(\unitary \times G)$ admits a Hamiltonian action of a high-dimensional torus, which considerably simplifies the task of understanding the symplectic measure.

If a $2n$-dimensional symplectic manifold $(M,\omega)$ admits a Hamiltonian action of a torus $U(1)^k$, then the associated momentum map $\Phi: M \to \left(\mathfrak{u}(1)^k\right)^\ast \simeq \R^k$ has convex image and connected level sets:

\begin{thm}[Atiyah~\cite{Atiyah:1982ih} and Guillemin--Sternberg~\cite{Guillemin:1982gx}] \label{thm:convexity}
	With notation as above:
	\begin{itemize}
		\item For any $v \in \R^k$, $\Phi^{-1}(v)$ is either empty or connected.
		\item The image $\Phi(M)$ is the convex hull of the images of the fixed points of the torus action.
	\end{itemize}
\end{thm}

In particular, when $M$ is (relatively) compact, (the closure of) $\Phi(M)$ is a bounded convex polytope $P$ called the \emph{moment polytope} associated to the Hamiltonian torus action.

In general, the Duistermaat--Heckman theorem~\cite{Duistermaat:1982hq} precisely describes the relationship between the pushforward measure $\Phi_\ast(m_\omega)$ and Lebesgue measure on the moment polytope $P$. Since we will not need the full statement, we restrict to the case when $k=n$, that is when $M$ admits a Hamiltonian action of a half-dimensional torus. In this case $M$ is called a \emph{toric symplectic manifold}. Toric symplectic manifolds are closely related to toric varieties~\cite{cox_toric_2011} and are completely classified by the combinatorics of the moment polytope~\cite{delzant_hamiltoniens_1988}.

\begin{thm}[Duistermaat–Heckman \cite{Duistermaat:1982hq}, see  also {\cite[Chapter~30]{CannasdaSilva:2001cg}}]\label{thm:duistermaat_heckman}
Let $M$ be a $2n$-dimensional toric symplectic manifold with moment polytope $P$. The pushforward measure $\Phi_\ast(m_\omega)$ is a constant multiple of Lebesgue measure on~$P$.
\end{thm}

\subsection{Toric Structure of Frame Space}

\subsubsection{Circle Actions on Frame Space}

With \Cref{thm:duistermaat_heckman} in mind, our strategy is to show that the top stratum of $\frames_{\blam}(\br)/(\unitary \times G)$ is toric, and then to see that the image of the spark-deficient frames in the moment polytope has measure zero with respect to Lebesgue measure. 

We begin by defining a torus action on the full space of frames $\frames$ and then show that our construction descends to the symplectic quotient. Let $F \in \frames$ with column vectors $f_i \in \C^{d \times 1}$, $i = 1, \ldots, N$. For each $k = 1, \ldots, N$, let 
\[
\mu_{k1} \geq \mu_{k2} \geq \cdots \geq \mu_{k d} \geq 0
\]
denote the $d$ eigenvalues of the $d \times d$ Hermitian matrix
\[
f_1 f_1^\ast + f_2 f_2^\ast + \cdots + f_k f_k^\ast,
\]
arranged in decreasing order---each such sum is rank at most $k$, so we only need consider the first $k$ eigenvalues if $k<d$. The quantities $\mu_{kj}$ were dubbed \emph{eigensteps} by Cahill et al.~\cite{Cahill:2013jv}. 

To get well-defined circle actions, we make the assumption that the eigenvalues $\mu_{kj}$, $j=1,\ldots,k$, are all distinct---the necessity of this assumption is explained below in \Cref{rmk:circle_action_well_defined}. Let $u_{k1},\ldots,u_{k\min(k,d)}$ be the eigenvectors corresponding to the nonzero eigenvalues. For each $k = 1,\ldots,N$ and each $j = 1,\ldots,\min(k,d)$, define a $\torus$-action on $F$, denoted
\begin{align*}
    \phi_{kj}: \torus \times \frames \to \frames,
\end{align*}
by defining it at the frame vector level as
\begin{equation}\label{eqn:circle_action}
\phi_{kj}(t,f_i) = \left\{\begin{array}{cc}\exp(t \sqrt{-1} u_{kj}u_{kj}^\ast) f_i & \mbox{if $1 \leq i \leq k$} \\
f_i & \mbox{if $k+1 \leq i \leq N$,} \end{array}\right.
\end{equation}
where $t \in [0,2\pi)$.

\begin{remark}
To be precise, we are identifying $[0,2\pi) \approx \torus$ via $t \mapsto \exp(t\sqrt{-1})$. This identification induces an isomorphism $\R \approx \mathfrak{u}(1)$ via
\begin{equation}\label{eqn:identify_circle_lie_algebra}
    s \mapsto \sqrt{-1}s.
\end{equation}
Keeping track of the exact isomorphism used in this identification will be useful later on.
\end{remark}

\begin{remark}\label{rmk:circle_action_well_defined}
The assumption that the eigenvalues $\mu_{kj}$ are distinct means that the eigenspaces are all one-dimensional, which, in turn, gives a well-defined ordering of the eigenvectors $u_{kj}$. Without the isolated eigenvalues assumption, some of these actions would degenerate to $\mathrm{U}(\ell)$-actions, with $\ell$ the multiplicity of a repeated eigenvalue.
\end{remark}

\begin{remark}
The action defined by \eqref{eqn:circle_action} really defines a circle action; i.e., it is $2\pi$-periodic. Indeed, this follows from \Cref{lem:exp projection} below, and is also shown in \cite[Corollary 5.1.4]{Flaschka:2005dq}.
\end{remark}

Next we show that this action on $\frames$ induces a well-defined action on $\frames_{\blam}(\br)/(\unitary \times G)$. We do so in stages.

\begin{prop}\label{prop:action torus invariant}
    The circle action $\phi_{kj}$ commutes with the $\bt$ action.
\end{prop}

\begin{proof}
    We need to show that if two frames $F_1$ and $F_2$ lie in the same $\bt$ orbit, then so do $\phi_{kj}(t,F_1)$ and $\phi_{kj}(t,F_2)$. Indeed, this holds since the formula for the action \eqref{eqn:circle_action} makes it clear that
    \[
    \phi_{kj}(t,f_i e^{-\sqrt{-1}\theta}) = \phi_{kj}(t,f_i) e^{-\sqrt{-1}\theta}
    \]
    for all $\theta$.
\end{proof}

\begin{prop}\label{prop:action unitary invariant for vectors}
    Let $F = \begin{bmatrix}f_1\mid \cdots \mid f_N\end{bmatrix} \in \frames$. The circle action $\phi_{kj}$ on the vectors of $F$ is $\unitary$-equivariant. That is, for all $A \in \unitary$, 
    \[
    \phi_{kj}(t,Af_i) = A \phi_{kj}(t,f_i)
    \]
    for all $t \in [0,2\pi)$ and all $i = 1,\ldots,N$.
\end{prop}

\begin{proof}
    The claim is clear when $k + 1 \leq i \leq N$, since the action is trivial in that case. It remains to check the claim when $1 \leq i \leq k$. In this case, we have
    \[
        \phi_{kj}(t,Af_i) = \exp(t\sqrt{-1}Au_{kj}u_{kj}^\ast A^\ast) A f_i,
    \]
    since $(\mu_{kj},Au_{kj})$ are the corresponding eigenpair for the $k$th partial frame operator of $AF$. In turn, we have 
    \[
        \exp(t\sqrt{-1}Au_{kj}u_{kj}^\ast A^\ast) f_i = A \exp(t\sqrt{-1}u_{kj}u_{kj}^\ast) A^\ast A f_i = A \exp(t\sqrt{-1}u_{kj}u_{kj}^\ast) f_i = A \phi_{kj}(t,f_i). \qedhere
    \]
\end{proof}

Since the circle action $\phi_{kj}$ is defined on $F = \begin{bmatrix}f_1\mid\cdots\mid f_N\end{bmatrix} \in \frames$ through the action on the individual frame vectors, we have the following corollary.

\begin{cor}\label{cor:action unitary invariant}
    The circle action $\phi_{kj}$ on $\frames$ commutes with the $\unitary$ action. That is, for all $A \in \unitary$, 
    \[
    \phi_{kj}(t,AF) = A \phi_{kj}(t,F)
    \]
    for all $t \in [0,2\pi)$ and all $F \in \frames$.
\end{cor}

The next proposition is obvious from the formula for $\phi_{kj}$.

\begin{prop}\label{prop:action preserves norms}
    The circle action $\phi_{kj}$ preserves norms of frame vectors.
\end{prop}

The following proposition is similar, but is less obvious.

\begin{prop}\label{prop:action preserves spectrum}
	The circle action $\phi_{kj}$ preserves the frame operator.
\end{prop}

The proof uses a lemma.

\begin{lem}\label{lem:exp projection}
	If $P$ is a $d \times d$ projection matrix, then $\exp(t\sqrt{-1}P) = \Id_d + (e^{t\sqrt{-1}} -1)P$.
\end{lem}

\begin{proof}
	This follows from the Taylor series representation of the matrix exponential and the property that $P^2 = P$.
\end{proof}

\begin{proof}[Proof of \Cref{prop:action preserves spectrum}]
	Suppose $F = \begin{bmatrix} f_1 \mid f_2 \mid \cdots \mid f_N \end{bmatrix} \in \C^{d \times N}$ is a frame and let $S = F F^* = f_1 f_1^* + \dots + f_N f_N^*$ be its frame operator with spectrum $\lambda_1 \geq \dots \geq \lambda_d >0$. Let $k \in \{1, \dots , N\}$ and let $S_k = f_1 f_1^* + \dots + f_k f_k^*$ be the partial frame operator with spectrum $\mu_{k1} \geq \dots \geq \mu_{kd} \geq 0$. Let $j \in \{1, \dots , \min(k,d)\}$, and consider the torus action associated with the eigenvalue $\mu_{kj}$. Then, by definition,
	\[
		t \cdot F = \begin{bmatrix} \exp(t \sqrt{-1} u_{kj}u_{kj}^\ast) f_1 \mid \cdots \mid \exp(t \sqrt{-1} u_{kj}u_{kj}^\ast) f_k \mid f_{k+1}\mid \cdots \mid f_N \end{bmatrix}
	\]
	and the corresponding frame operator is
	\begin{align*}
		S_t & := (t \cdot F)(t \cdot F)^* = \exp(t \sqrt{-1} u_{kj}u_{kj}^\ast)(f_1 f_1^* + \dots + f_kf_k^*)\exp(-t \sqrt{-1} u_{kj}u_{kj}^\ast) + f_{k+1}f_{k+1}^* + \dots + f_N f_N^* \\
		& = (\Id_d + (e^{t\sqrt{-1}} -1)u_{kj}u_{kj}^*)S_k(\Id_d + (e^{-t\sqrt{-1}} -1)u_{kj}u_{kj}^*) + f_{k+1}f_{k+1}^* + \dots + f_N f_N^* \\
		& = S + (e^{t\sqrt{-1}} -1)u_{kj}u_{kj}^*S_k + S_k(e^{-t\sqrt{-1}} -1)u_{kj}u_{kj}^* + (2-e^{t\sqrt{-1}}-e^{-t\sqrt{-1}})u_{kj}u_{kj}^*S_ku_{kj}u_{kj}^*,
	\end{align*}
	using \Cref{lem:exp projection}. 
	
	For $i_1 \neq i_2$ the vectors $u_{ki_1}$ and $u_{ki_2}$ are Hermitian orthogonal and hence the product of projections $u_{ki_1}u_{ki_1}^* u_{ki_2}u_{ki_2}^*=0$. Combining this with the spectral decomposition
	\[
		S_k = \mu_{k1}u_{k1}u_{k1}^* + \dots + \mu_{kd}u_{kd}u_{kd}^*,
	\]
	and again using the fact that projections are idempotent, we see that
	\[
		S_t - S = \left[(e^{t\sqrt{-1}} -1) +(e^{-t\sqrt{-1}} -1) + (2-e^{t\sqrt{-1}}-e^{-t\sqrt{-1}})\right]\mu_{kj}u_{kj}u_{kj}^* = 0,
	\]
	so the frame operator is invariant under the circle action.
\end{proof}

Combining \Cref{prop:action torus invariant,prop:action preserves norms,prop:action preserves spectrum} with \Cref{cor:action unitary invariant}, we obtain:

\begin{prop}
    The circle action $\phi_{kj}$ on $\frames$ descends to a well-defined action on $\frames_{\blam}(\br)/(\unitary \times G)$.
\end{prop}

\subsubsection{The Momentum Map of the Circle Action}

For each $k = 1,\ldots,N$ and $j = 1,\ldots,\min(k,d)$, define a map
\begin{align*}
\Phi_{kj}:\frames &\to \R \\
F &\mapsto \mu_{kj},
\end{align*}
where $\mu_{kj} = \mu_{kj}(F)$ is (as above) the $j$th eigenvalue (in descending order) of the $k$th partial frame operator of $F$. Recall that we identify $\mathfrak{u}(1) \approx \R$ via \eqref{eqn:identify_circle_lie_algebra}. We likewise identify $\mathfrak{u}(1)^\ast \approx \R$, where $\R^\ast \approx \R$ is identified via the pairing
\begin{equation}\label{eqn:pairing_circle}
\langle s,t \rangle  = \frac{s \cdot t}{2}
\end{equation}
on $\R \times \R$. This allows us to state the following result.

\begin{prop}\label{prop:circle_action_momentum}
    The map $\Phi_{kj}:\frames \to \R$ is a momentum map for the circle action $\phi_{kj}$ on the dense open subset of matrices whose $k$th partial frame operators have isolated $j$th eigenvalues.
\end{prop}

To prove the proposition, we introduce some notation and technical lemmas. Let $F = \begin{bmatrix}f_1\mid \cdots \mid f_N\end{bmatrix} \in \frames$ and let $S_k = S_k(F) = f_1 f_1^\ast + \cdots + f_k f_k^\ast$ be its partial frame operator. Then $\mu_{kj}(F)$ is the $j$th eigenvalue (in decreasing order) of $S_k$.

We first consider the map $\bar{\mu}_j$ which takes a Hermitian $k\times k$ matrix to its $j$th eigenvalue. In what follows, let $\langle \cdot,\cdot\rangle$ denote the Frobenius inner product on $\frames$. 

\begin{lem}\label{lem:jth_eigenvalue_derivative}
    Suppose that $S$ is a $k\times k$ Hermitian matrix with isolated $j$th eigenvalue $\bar{\mu}_j(S)$. Then $\bar{\mu}_j$ is a smooth map in a neighborhood of $S$. Moreover, its gradient at $S$ is given by 
    \[
    \nabla \bar{\mu}_j (S) = \bar{u}_j \bar{u}_j^\ast,
    \]
   where $\bar{u}_j$ is the unit eigenvector associated to $\bar{\mu}_j(S)$.
\end{lem}

\begin{proof}
Let $S$ have isolated $j$th eigenvalue, which we denote simply as $\mu$. We denote the associated unit eigenvector as $u$. Let $S'$ be a variation of $S$. Denote the $j$th eigenvalue and unit eigenvector of $S + \epsilon S'$ as $\mu(\epsilon)$ and $u(\epsilon)$, respectively; we denote the $\epsilon$-derivatives of these functions as $\dot{\mu}(\epsilon)$ and $\dot{u}(\epsilon)$. Then
\begin{align}
    \left.\frac{d}{d\epsilon}\right|_{\epsilon} \mu(\epsilon) &= \left.\frac{d}{d\epsilon}\right|_{\epsilon} \langle (S + \epsilon S')u(\epsilon), u(\epsilon)\rangle = \langle S' u, u \rangle + \langle S \dot{u}(0), u \rangle + \langle Su, \dot{u}(0) \rangle \nonumber \\
    &= \langle S' u, u \rangle + \langle \dot{u}(0), Su \rangle + \langle Su, \dot{u}(0) \rangle \label{eqn:derivative_proof_1} \\
    &= \langle S' u, u \rangle + \mu\left(\langle \dot{u}(0), u \rangle + \langle u, \dot{u}(0)\right) \rangle \label{eqn:derivative_proof_2} \\
    &= \langle S'u,u \rangle = \langle S', uu^\ast \rangle, \nonumber
\end{align}
where \eqref{eqn:derivative_proof_1} follows because $S$ is Hermitian and \eqref{eqn:derivative_proof_2} follows by the condition that $\|u(\epsilon)\|$ is constant in $\epsilon$.
\end{proof}

\begin{lem}\label{lem:infinitesimal_vector_field}
    Let $F = \begin{bmatrix}f_1\mid \cdots \mid f_N\end{bmatrix} \in \frames$ and $s \in \R \approx \mathfrak{u}(1)$ (using the identification \eqref{eqn:identify_circle_lie_algebra}). The infinitesimal vector field associated to $s$ which is induced by the action $\phi_{kj}$ is given at $F$ by
    \[
    F^s := \sqrt{-1} s \cdot u_{kj}u_{kj}^\ast \begin{bmatrix} f_1 \mid \cdots \mid f_k \mid 0 \mid \cdots \mid 0 \end{bmatrix}.
    \]
\end{lem}

\begin{proof}
The infinitesimal vector field is given at the frame level by
\begin{align*}
\left. \frac{d}{d\epsilon} \right|_{\epsilon = 0} \phi_{kj}(\epsilon s, f_i) 
&= 
\left. \frac{d}{d\epsilon} \right|_{\epsilon = 0} \left\{\begin{array}{cc}
\exp(\epsilon s \sqrt{-1} u_{kj}u_{kj}^\ast) f_i & \mbox{if $1 \leq i \leq k$} \\
f_i & \mbox{if $k+1 \leq i \leq N$,}
\end{array}\right.\\
&= 
\left\{\begin{array}{cc}
\sqrt{-1}s\cdot u_{kj} u_{kj}^\ast f_i & \mbox{if $1 \leq i \leq k$} \\
0 & \mbox{if $k+1 \leq i \leq N$.}
\end{array}\right.
\end{align*}
The result follows.
\end{proof}

We are now prepared to prove the proposition.

\begin{proof}[{Proof of \Cref{prop:circle_action_momentum}}]
We need to show that, for $F = \begin{bmatrix}f_1\mid \cdots \mid f_N\end{bmatrix} \in \frames$, $X = \begin{bmatrix}x_1 \mid \cdots \mid x_N\end{bmatrix} \in T_F \frames \approx \C^{d \times N}$, and $s \in \R \approx \mathfrak{u}(1)$ (via \eqref{eqn:identify_circle_lie_algebra}), the following equation holds:
\begin{equation}\label{eqn:momentum_equation_circle}
D\Phi_{kj}(F)(X)(s) = \omega_F(F^s,X).
\end{equation}
This is accomplished by direct computation.

The right hand side of \eqref{eqn:momentum_equation_circle} simplifies to
\begin{align}
    \omega_F(F^s,X) &= \mathrm{Im}\,\mathrm{tr}(F^s X^\ast) \nonumber \\
    &= \mathrm{Im} \, \mathrm{tr} \left(\sqrt{-1}s \cdot u_{kj} u_{kj}^\ast \cdot \begin{bmatrix}f_1\mid \cdots \mid f_k \mid 0 \mid \cdots \mid 0 \,\end{bmatrix} \cdot \begin{bmatrix}\, x_1\mid \cdots \mid x_N\end{bmatrix}^\ast \right) \label{eqn:momentum_map_RHS_1} \\
    &= \mathrm{Im} \, \sqrt{-1}s \cdot  \mathrm{tr}\left(u_{kj}u_{kj}^\ast (f_1x_1^\ast + \cdots + f_k x_k^\ast ) \right) \nonumber \\
    &= s \cdot \mathrm{Re} \, \mathrm{tr}\left(u_{kj}u_{kj}^\ast (f_1x_1^\ast + \cdots + f_k x_k^\ast ) \right) \label{eqn:momentum_map_RHS_2},
\end{align}
where \eqref{eqn:momentum_map_RHS_1} uses the expression for $F^s$ from \Cref{lem:infinitesimal_vector_field}. The left hand side of \eqref{eqn:momentum_equation_circle} becomes
\begin{align*}
    D\Phi_{kj}(F)(X)(s) &= \frac{1}{2} s \cdot D(\bar{\mu}_j \circ S_k)(F)(X) \\
    &= \frac{1}{2} s \cdot D\bar{\mu}_j(S_k(F)) \cdot DS_k(F)(X),
\end{align*}
where the first line follows by our choice of identification $\R \approx \mathfrak{u}(1)^\ast$ \eqref{eqn:pairing_circle} and where $S_k(F)$ is the $k$th partial frame operator of $F$ and $\bar{\mu}_j$ is the $j$th eigenvalue function on $d \times d$ Hermitian matrices. A straightforward computation shows that
\[
DS_k(F)(X) = (f_1 x_1^\ast + x_1 f_1^\ast) + \cdots + (f_k x_k^\ast + x_k f_k^\ast).
\]
Putting this together with \Cref{lem:jth_eigenvalue_derivative}, our simplification continues as
\begin{align}
    \frac{1}{2} s \cdot D\bar{\mu}_j(S_k(F)) \cdot DS_k(F)(X) &= \frac{1}{2} s \cdot \mathrm{Re} \, \langle u_{kj} u_{kj}^\ast, (f_1 x_1^\ast + x_1 f_1^\ast) + \cdots + (f_k x_k^\ast + x_k f_k^\ast) \rangle \nonumber \\
    &= \frac{1}{2} s \cdot \mathrm{Re} \, \mathrm{tr}\left(u_{kj} u_{kj}^\ast \left( (f_1 x_1^\ast + x_1 f_1^\ast) + \cdots + (f_k x_k^\ast + x_k f_k^\ast) \right)\right) \nonumber \\
    &= \frac{1}{2} s \cdot \mathrm{Re} \, \mathrm{tr}\left(u_{kj} u_{kj}^\ast \left(f_1 x_1^\ast + f_2 x_2^\ast + \cdots + f_k x_k^\ast \right) \right)  \nonumber \\
    &\hspace{2in} + \mathrm{Re}\,\mathrm{tr}\left(u_{kj} u_{kj}^\ast \left(x_1 f_1^\ast + x_2 f_2^\ast + \cdots + x_k f_k^\ast \right) \right) \label{eqn:momentum_LHS_1} \\
    &= s \cdot \mathrm{Re} \, \mathrm{tr}\left(u_{kj} u_{kj}^\ast \left(f_1 x_1^\ast + \cdots + f_k x_k^\ast \right) \right), \label{eqn:momentum_LHS_2}
\end{align}
where the last line follows by the observation that the trace terms in \eqref{eqn:momentum_LHS_1} are conjugates. Since \eqref{eqn:momentum_LHS_2} and \eqref{eqn:momentum_map_RHS_2} agree, \eqref{eqn:momentum_equation_circle} has been established.
\end{proof}

\subsubsection{Dimension Counting}\label{sec:dimension_counting}

We now pause briefly to count the possible number of independent eigensteps $\mu_{kj}$. The $\mu_{kj}$ associated to a frame define a map $\frames_{\blam}(\br) \to \R^m$ for some $m$ depending on $N$ and $d$; as has been previously observed~\cite{Cahill:2013jv,Cahill:2017gv,Haga:2016jva}, the image of this map is a convex polytope, whose dimension we now determine. To slightly simplify some calculations, we will assume in this section that $\br$ is strongly $\blam$-admissible.

Since each partial frame operator $S_k =  f_1f_1^* + \dots + f_k f_k^*$ is a rank-1 perturbation $S_k = S_{k-1} + f_k f_k^*$ of the previous partial frame operator, Weyl's perturbation inequalities (see, e.g., \cite[Chapter III]{Bhatia}) imply that the eigenvalues of $S_k$ and $S_{k-1}$ satisfy the interlacing inequalities
\[
	\dots \geq \mu_{k,j} \geq \mu_{k-1,j} \geq \mu_{k,j+1} \geq \mu_{k-1,j+1} \geq \dots
\]
for each $k=2, \dots , N$ and $j=1, \dots , \min\{d,k\}$, where we use the notation $\mu_{k,j}$ in place of $\mu_{kj}$ when we need to disambiguate the index $k$ from the index $j$.

When $N=4$ and $d=3$, this produces the array
\[
	\begin{tikzcd}[row sep=small,column sep=tiny]
	\lambda_1 && \lambda_2 && \lambda_3 \\
		& \mu_{31} && \mu_{32} && \mu_{33} \\
		&& \mu_{21} && \mu_{22} \\
		&&& \mu_{11}
		\arrow["\geq"{marking}, draw=none, from=2-2, to=3-3]
		\arrow["\geq"{marking}, draw=none, from=3-3, to=4-4]
		\arrow["\geq"{marking}, draw=none, from=4-4, to=3-5]
		\arrow["\geq"{marking}, draw=none, from=3-5, to=2-6]
		\arrow["\geq"{marking}, draw=none, from=3-3, to=3-5]
		\arrow["\geq"{marking}, draw=none, from=3-3, to=2-4]
		\arrow["\geq"{marking}, draw=none, from=2-2, to=2-4]
		\arrow["\geq"{marking}, draw=none, from=2-4, to=2-6]
		\arrow["\geq"{marking}, draw=none, from=2-4, to=3-5]
		\arrow["\geq"{marking}, draw=none, from=1-1, to=2-2]
		\arrow["\geq"{marking}, draw=none, from=2-2, to=1-3]
		\arrow["\geq"{marking}, draw=none, from=1-1, to=1-3]
		\arrow["\geq"{marking}, draw=none, from=1-3, to=1-5]
		\arrow["\geq"{marking}, draw=none, from=1-3, to=2-4]
		\arrow["\geq"{marking}, draw=none, from=2-4, to=1-5]
		\arrow["\geq"{marking}, draw=none, from=1-5, to=2-6]
	\end{tikzcd}
\]
where we recall that $S_N = S$, and hence $\mu_{N,j} = \lambda_j$ for each $j=1, \dots , d$.

In other words, entries in the array are greater than entries to the right, regardless of vertical position. With this convention in place, we can omit explicit inequalities without losing any information:
\[
	\begin{tikzcd}[cramped, sep=0]
	\lambda_1 && \lambda_2 && \lambda_3 \\
		& \mu_{31} && \mu_{32} && \mu_{33} \\
		&& \mu_{21} && \mu_{22} \\
		&&& \mu_{11}
	\end{tikzcd}
\]

In general, then, the $\mu_{k,j}$ satisfy the system of interlacing inequalities implied by the following diagram:	
\begin{equation}\label{eq:eigenstep inequalities}
	\begin{tikzcd}[cramped, sep=0]
	{\lambda_1} && {\lambda_2} && {\lambda_3} && \dots && {\lambda_d} \\
	& {\mu_{N-1,1}} && {\mu_{N-1,2}} && {\mu_{N-1,3}} && \dots && {\mu_{N-1,d}} \\
	&& \ddots && \ddots && \ddots && \ddots && \ddots \\
	&&& {\mu_{d+1,1}} && {\mu_{d+1,2}} && {\mu_{d+1,3}} && \dots && {\mu_{d+1,d}} \\
	&&&& {\mu_{d,1}} && {\mu_{d,2}} && {\mu_{d,3}} && \dots && {\mu_{d,d}} \\
	&&&&& \ddots && \ddots && \ddots && \iddots \\
	&&&&&& {\mu_{3,1}} && {\mu_{3,2}} && {\mu_{3,3}} \\
	&&&&&&& {\mu_{2,1}} && {\mu_{2,2}} \\
	&&&&&&&& {\mu_{1,1}}
	\end{tikzcd}
\end{equation}
In this diagram, all entries are nonnegative; equivalently, we think of implicit zeros to the right of every row. The pattern of inequalities described by the above diagram is often called a \emph{Gelfand--Tsetlin pattern}~\cite{gelfand_finite-dimensional_1950,DeLoera:2004eo}; it corresponds to the Gelfand--Tsetlin integrable system on the flag manifold $\mathcal{O}_{\widetilde{\blam}} \simeq \frames_{\blam}/\unitary$~\cite{Anonymous:1983bj}.

In fact, there is one further constraint: cyclic invariance of trace implies that the trace of each partial frame operator $S_k$ must equal the trace of the corresponding partial Gram matrix. That is, the sum of each row in \eqref{eq:eigenstep inequalities} must equal the corresponding partial sum of squared frame vector norms:
\begin{equation}\label{eq:row sum}
	\sum_{j=1}^{\min\{k,d\}} \mu_{kj} = \sum_{j=1}^k r_j
\end{equation}
for each $k=1, \dots , N$.

With this constraint in place, we can count the number of free parameters in~\eqref{eq:eigenstep inequalities} and, hence, the number of independent $\phi_{kj}$. Notice that there are $\frac{d(2N-d-1)}{2}$ total entries in~\eqref{eq:eigenstep inequalities}, excluding the top row (which we already know is fixed): $\frac{d(d+1)}{2}$ for the triangle in the bottom $d$ rows, and $d(N-d-1)$ for the parallelogram in the upper $N-d$ rows. Since each row sum is fixed, we lose one free parameter for each of the rows but the top one, meaning we subtract $N-1$ parameters.

If all the eigenvalues $\lambda_1, \dots , \lambda_d$ of the frame operator are distinct, we are done. However, if $\lambda_j=\lambda_{j+1}$, then this implies that $\lambda_j=\mu_{N-1,j}=\lambda_{j+1}$. More generally, if $\lambda_j$ has multiplicity $k$, meaning that $\lambda_j=\lambda_{j+1} = \dots = \lambda_{j+k-1}$, then this fixes an entire upside-down triangle in~\eqref{eq:eigenstep inequalities} with vertices at $\lambda_j$, $\lambda_{j+k-1}$, and $\mu_{N-k,j}$, comprising $\frac{k(k-1)}{2}$ of the $\mu_{i,j}$.

Therefore, if the eigenvalues $\lambda_1, \dots , \lambda_d$ have multiplicities $k_1, \dots , k_\ell$ with $k_1 + \dots + k_\ell=d$, then the total number of free parameters in~\eqref{eq:eigenstep inequalities} is
\[
	\frac{d(2N-d-1)}{2} - (N-1) - \sum_{j=1}^\ell \frac{k_j(k_j-1)}{2} = N(d-1)+1-\frac{d^2}{2}-\frac{1}{2}\sum_{j=1}^\ell k_j^2.
\]

The interlacing inequalities implicit in the diagram \eqref{eq:eigenstep inequalities} together with the row sums \eqref{eq:row sum} determine a convex polytope $\polytope_{\blam}(\br)$, which we call the \emph{eigenstep polytope} (at least when $\blam$ and $\br$ are rational, these are sometimes called \emph{(weight-restricted) Gelfand--Tsetlin polytopes} in the combinatorics literature~\cite{DeLoera:2004eo,Alexandersson:2016gb}). In the course of the discussion above, we have proved:

\begin{prop}\label{prop:polytope dimension}
	When $\br$ is strongly $\blam$-admissible, the eigenstep polytope has dimension
	\[
		d(\blam,\br):= \dim\left(\polytope_{\blam}(\br)\right) =  N(d-1)+1-\frac{d^2}{2}-\frac{1}{2}\sum_{j=1}^\ell k_j^2,
	\]
	where $k_1, \dots , k_\ell$ are the multiplicities of the spectrum $\blam = (\lambda_1, \dots , \lambda_d)$.
	
	Comparing to \Cref{cor:dimension}, 
	\[
		\dim\left(\polytope_{\blam}(\br)\right) = \frac{1}{2} \dim\left(\frames_{\blam}(\br)/(\unitary \times G)\right).
	\]
\end{prop}

In the case of unit-norm tight frames, $r_i=1$ for all $i$ and $\lambda_j=\frac{N}{d}$ for all $j$, so $k_1=d$ and the above dimension simplifies to
\[
	\dim\left(\polytope_{\left(\frac{N}{d}, \dots , \frac{N}{d}\right)}(1,\dots , 1)\right) = (d-1)(N-d-1),
\]
so we see that a result of Flaschka--Millson~\cite[Proposition~6.2.2]{Flaschka:2005dq} and Haga--Pegel~\cite[Theorem~3.2]{Haga:2016jva} is a special case of \Cref{prop:polytope dimension}.

\subsubsection{Hamiltonian Torus Action}

We now return to the task of showing that there is a Hamiltonian torus action on our frame spaces; the dimension $d(\blam,\br)$ just computed will be the dimension of the torus. 

Of course, the torus action will simply be the product of the individual circle actions, which we will show commute using \Cref{prop:action unitary invariant for vectors}.

\begin{prop}
    The circle actions $\phi_{kj}$ and $\phi_{m \ell}$ commute for all $k,j,m,\ell$. 
\end{prop}

\begin{proof}
Let $F = \begin{bmatrix}f_1\mid \cdots \mid f_N\end{bmatrix} \in \frames$. It suffices to prove the claim at the frame level, i.e., 
\begin{equation}\label{eqn:circles_commute}
\phi_{kj}(t,\phi_{m \ell}(s,f_i)) = \phi_{m \ell}(s,\phi_{kj}(t,f_i))
\end{equation}
for all $f_i$. Without loss of generality, assume $k \leq m$. If $k+1 \leq i \leq N$, then $\phi_{kj}$ is the trivial action, so \eqref{eqn:circles_commute} is obvious. Otherwise we have
\begin{align}
\phi_{kj}(t,\phi_{\ell m}(s,f_i)) &= \phi_{kj}(t, \exp(s \sqrt{-1} u_{m \ell} u_{m \ell}^\ast) f_i ) \nonumber \\
&= \exp(s \sqrt{-1} u_{m \ell} u_{m \ell}^\ast) \phi_{kj}(t,f_i) \label{eqn:circles_commute_1} \\
&= \phi_{m\ell} (s, \phi_{kj}(t,f_i)), \nonumber
\end{align}
where \eqref{eqn:circles_commute_1} follows by \Cref{prop:action unitary invariant for vectors}, since the exponential is unitary.
\end{proof}

\begin{cor}\label{cor:open_dense_torus}
If $\br$ is strongly $\blam$-admissible, there is a Hamiltonian action of the torus $U(1)^{d(\blam,\br)}$ on an open, dense subset of $\frames_{\blam}(\br)/(\unitary \times G)$. The associated moment polytope is the eigenstep polytope $\polytope_{\blam}(\br)$.
\end{cor}

\begin{proof}
    We know from \Cref{cor:proper} that $\frames_{\blam}(\br)/(\mathrm{U}(d) \times G)$ contains an open, dense symplectic manifold which, by \Cref{cor:when frame spaces are manifolds}, consists of orbits of non-orthodecomposable frames. This symplectic manifold contains an open, dense subset of frames whose partial frame operators each have as many distinct eigenvalues as possible (given the eigenstep constraints). Choosing $d(\blam,\br)$ free variables according to the dimension counting argument of \Cref{sec:dimension_counting} yields a Hamiltonian torus action on this open set.
\end{proof}

\begin{remark}
	In conjunction with \Cref{thm:convexity}, this result can be used to show that the space of frames with prescribed eigensteps is connected. While this fact essentially follows from~\cite[Theorem~7]{Cahill:2013jv}, which was a key tool in the original proof of the Frame Homotopy Conjecture~\cite{Cahill:2017gv}, the symplectic viewpoint puts this fact in a more general context.
\end{remark}

By \Cref{prop:polytope dimension}, the torus acting on the open dense subset from \Cref{cor:open_dense_torus} is half-dimensional. We have therefore proved the following theorem on the geometric structure of frame spaces.

\begin{thm}\label{thm:toric_symplectic}
The space of frames $\frames_{\blam}(\br)$ with prescribed spectrum and norms contains a dense open subset which is a $(\unitary \times G)$-bundle over a toric symplectic manifold.
\end{thm}

\section{Proof of Main Theorem}\label{sec:proof_of_main_theorem}

We are now prepared to prove the main theorem using the toric symplectic framework developed in the previous section. We also sketch a more algebraically-flavored proof using related tools from Geometric Invariant Theory \cite{mumford_geometric_1994}.

\subsection{A Toric Symplectic Approach}

It is clear that the three options (repeated below) are mutually exclusive and exhaust all possibilities, so it remains only to show that the conclusion in each part follows from the assumption.

\begin{enumerate}[leftmargin=*]
	\item \emph{If $\br$ is not $\blam$-admissible, then $\frames_{\blam}(\br) = \emptyset$.}
	\begin{proof}
		As pointed out in \Cref{sec:prescribed norms and spectrum}, this is a consequence of the Schur--Horn theorem~\cite{schur1923uber,horn1954doubly}; it is also a result of Casazza and Leon~\cite{Casazza:2010ti}.
	\end{proof}
	\item \emph{If $\br$ is $\blam$-admissible, but not strongly $\blam$-admissible, then $\frames_{\blam}(\br)$ is nonempty but consists entirely of frames which are not full spark.}
	\begin{proof}
		Suppose $\br$ is $\blam$-admissible, but not strongly $\blam$-admissible. Then $\frames_{\blam}(\br)$ is nonempty by Casazza and Leon's result~\cite{Casazza:2010ti}. 
		
		Since $\br$ is not strongly $\blam$-admissible, there exists $k \in \{1, \dots , d\}$ so that 
		\[
			\sum_{i=1}^k r_i = \sum_{i=1}^k \lambda_i.
		\]
		The admissibility criterion requires that $\sum_{i=1}^N r_i = \sum_{i=1}^d \lambda_i$, so it follows that
		\begin{equation}\label{eq:reversed norm sum}
			\sum_{i=k+1}^N r_i = \sum_{i=k+1}^d \lambda_i.
		\end{equation}
		If $k=d$, then we see that $r_{k+1} = \dots = r_N = 0$, which does not satisfy the hypothesis that the $r_i$ are all positive. 
	
		Otherwise, suppose $F = \begin{bmatrix}f_1\mid \cdots \mid f_N\end{bmatrix} \in \frames_{\blam}(\br)$, so that $\blam$ is the spectrum of $FF^\ast$ and $\|f_i\|^2 = r_i$ for all $i=1, \dots , N$. Form a new frame $\widetilde{F} = \begin{bmatrix}f_N\mid \cdots \mid f_1\end{bmatrix}$ by reversing the order of the columns of $F$. If $\widetilde{\mu}_{\ell,j}$ are the eigensteps of $\widetilde{F}$, then we know that
		\[
			\sum_{j=1}^{\ell} \widetilde{\mu}_{\ell,j} = \sum_{j=1}^{\ell}r_{N+1-j} = \sum_{i=N+1-\ell}^N r_i
		\]
		for all $\ell=1,\dots, N-1$. In particular, letting $\ell=N-k$ yields
		\begin{equation}\label{eq:reversed eigenstep sum}
			\sum_{j=1}^{N-k} \widetilde{\mu}_{N-k,j} = \sum_{i=k+1}^N r_i = \sum_{i=k+1}^d \lambda_i
		\end{equation}
		using \eqref{eq:reversed norm sum}. 
	
		Now, consider the portion of the eigenstep inequality diagram~\eqref{eq:eigenstep inequalities} for $\widetilde{F}$ starting from the $(N-k)$th row:
		\[
			\begin{tikzcd}[cramped, sep=0]
				{\lambda_1} && {\lambda_2} && {\lambda_3} && \dots && {\lambda_d} \\
				& {\widetilde{\mu}_{N-1,1}} && {\widetilde{\mu}_{N-1,2}} && {\widetilde{\mu}_{N-1,3}} && \dots && {\widetilde{\mu}_{N-1,d}} \\
				&& {\widetilde{\mu}_{N-2,1}} && {\widetilde{\mu}_{N-2,2}} && {\widetilde{\mu}_{N-2,3}} && \dots && {\widetilde{\mu}_{N-2,d}} \\
				&&& \ddots && \ddots && \ddots && \ddots && \ddots \\
				&&&& {\widetilde{\mu}_{N-k,1}} && {\widetilde{\mu}_{N-k,2}} && {\widetilde{\mu}_{N-k,3}} && \dots && {\widetilde{\mu}_{N-k,d}} \\
			\end{tikzcd}
		\]
		Going diagonally up and to the right, we see that $\widetilde{\mu}_{N-k,i} \geq \lambda_{k+i}$ for all $i=1,\dots , d-k$. The only way this can be reconciled with \eqref{eq:reversed eigenstep sum} is if
		\[
			\widetilde{\mu}_{N-k,1} = \lambda_{k+1},\,\dots ,\, \widetilde{\mu}_{N-k,d-k} = \lambda_{d},
		\]
		and hence
		\[
			\widetilde{\mu}_{N-k,d-k+1} = \dots = \widetilde{\mu}_{N-k,d} = 0.
		\]
		But then $\widetilde{\mu}_{N-k,d} \geq \widetilde{\mu}_{d,d}$, so it follows that $\widetilde{\mu}_{d,d}=0$. Since this is an eigenvalue of the partial frame operator
		\[
			f_N f_N^\ast + \dots + f_{N-d+1} f_{N-d+1}^\ast, 
		\]
		the length-$d$ collection of frame vectors $f_{N-d+1}, \dots , f_N$ is linearly dependent, and hence $F$ is spark-deficient. Since the choice of $F$ was arbitrary, we see that $\frames_{\blam}(\br)$ consists entirely of spark-deficient frames.
	\end{proof}
	
	\item \emph{If $\br$ is strongly $\blam$-admissible, then full spark frames have full measure in $\frames_{\blam}(\br)$.}
	
    \begin{proof}
	It suffices to show that the set of full spark frames is full measure in the dense open subset of $\frames_{\blam}(\br)$ from \Cref{thm:toric_symplectic}, which we denote $\widetilde{\mathcal{U}}$. We will prove a subclaim: the set $\mathcal{D}^d_{\blam}(\br) \subset \widetilde{\mathcal{U}}$ of frames  whose first $d$ columns are linearly dependent is measure zero. The full claim then follows easily. Indeed, consider the action of the symmetric group on $N$ letters on the matrix space $\C^{d \times N}$ given by permuting columns. This representation of the symmetric group embeds it as a subgroup of $\operatorname{U}(N)$, so that the action on $\C^{d\times N}$ is by isometries. The action of a permutation $\sigma$ restricts to an isometry of each frame space $\frames_{\blam}(\br)$ onto its image  $\frames_{\blam}(\sigma \cdot \br)$, where $\sigma \cdot \br$ is the corresponding permutation of the norm vector. It follows that the set of spark-deficient frames is realized as $\cup_{\sigma} \mathcal{D}^d_{\blam}(\sigma \cdot \br)$, a finite union of measure zero sets. We proceed by establishing the subclaim.
	
	A frame $F = \begin{bmatrix}f_1\mid \cdots \mid f_N\end{bmatrix} \in \widetilde{\mathcal{U}}$ has linearly independent columns $f_1,\ldots,f_d$ if and only if the partial frame operator $S_d = f_1 f_1^\ast + \cdots + f_d f_d^\ast$ is full rank, which holds if and only if the smallest eigenvalue of $S_d$ is positive. Observe that these conditions are well-defined on the $(\operatorname{U}(d) \times G)$-equivalence class of $F$, denoted $[F]$. Moreover, the quotient map $\widetilde{\mathcal{U}} \to \mathcal{U} \subset \frames_{\blam} (\br)/(\operatorname{U}(d) \times G)$ is a Riemannian submersion (with respect to the real parts of the respective K\"{a}hler structures) onto its (dense, open) image.  It is therefore sufficient to establish the subclaim for the set $\mathcal{U}$.
	
	Let 
	\[
	\Phi =  \bigtimes_{\substack{1 \leq k \leq N-1 \\ 1 \leq j \leq \min\{d,k\}}} \Phi_{kj} : \mathcal{U} \to \mathcal{P}_{\blam}^{d,N}(\br)
	\]
	denote the moment map for the torus action on $\mathcal{U}$; that is, $\Phi$ is the eigenstep map. By the remarks in the previous paragraph, the condition that the first $d$ columns of a frame $F$ are linearly independent is equivalent to the statement that $\Phi_{dd}([F]) \neq 0$. 
	
	According to \Cref{sec:dimension_counting}, the moment polytope $\mathcal{P}_{\blam}^{d,N}(\br)$ is a convex subset of an affine subspace $\mathcal{A} \subset \R^{\frac{d(2N-d-1)}{2}}$. Thinking of the $\Phi_{kj}$ as coordinates on the ambient space, the condition $\Phi_{dd} = 0$ defines a hyperplane. The intersection of this hyperplane with $\mathcal{P}_{\blam}^{d,N}(\br)$ is positive codimension (hence measure zero), unless $\mathcal{A}$ is contained in the hyperplane. It is easy to see that this is not the case, as it suffices to show the existence of a pattern of the form \eqref{eq:eigenstep inequalities} satisfying the defining equalities of $\mathcal{A}$ such that $\mu_{dd} \neq 0$---such examples are trivial to construct without the polytope inequality constraints. 
	
	We have so far shown that image of the set of equivalence classes of frames whose first $d$ columns are linearly independent has full measure image in $\mathcal{P}_{\blam}^{d,N}(\br)$ under $\Phi$. By the Duistermaat--Heckman Theorem (\Cref{thm:duistermaat_heckman}), this implies that the set has full measure in $\mathcal{U}$. This completes the proof of the subclaim, and therefore completes the proof of the theorem.
	\end{proof}
	
\end{enumerate}

\subsection{An algebraic approach} 
\label{sub:algebraic approach}

The above proof is based on symplectic geometry, but, under additional rationality assumptions, there is also an argument using algebraic geometry that shows that the collection of full-spark frames is dense in $\frames_{\blam}(\br)$.

While we expect a similar argument to apply to arbitrary rational $\blam$ and $\br$, in which case the objects of interest are more general weight varieties~\cite{Anonymous:1996uu}, for simplicity we limit ourselves to the case when $\blam$ and $\br$ are constant, so we are talking about (scaled) unit-norm tight frames. It will be convenient to rescale so that $\br = (d, \dots , d)$, and hence $\blam = (N, \dots , N)$; note that both are integer vectors. 

As above, it suffices to show that the (equivalence classes of) full spark frames are open and dense in the symplectic reduction 
\[
	\frames_{\blam}(\br)/(\operatorname{U}(d) \times G) \approx \C^{d \times N} \sslash_{\mathcal{O}_{-\blam} \times \{\br\}} (\unitary \times G).
\]
Since $\blam$ is constant, the coadjoint orbit $\mathcal{O}_{-\blam}$ consists of the single point $-\blam \Id_d$. Taking the reduction in stages, the above quotient is isomorphic to
\[
	\left(\C^{d \times N} \sslash_{-\blam \Id_d} \unitary \right)\sslash_{\br} G= \Gr\sslash_{\br} G
\]
using \Cref{prop:diffeomorphic_to_flag} and \Cref{remark:flag}.

Work of Sjamaar~\cite{sjamaar_holomorphic_1995}, which is a culmination of a series of results by Kempf and Ness~\cite{kempf_length_1979}, Guillemin and Sternberg~\cite{guillemin_geometric_1982}, Kirwan~\cite{Kirwan:1984jt}, and Ness~\cite{ness_stratification_1984} relating symplectic reductions and Geometric Invariant Theory (GIT) quotients~\cite{mumford_geometric_1994}, has the following consequence in our case:

\begin{thm}\label{thm:Reduction and GIT}
	For arbitrary $\br$ consisting of positive integers, $\Gr\sslash_{\br} G$ is isomorphic, as a complex projective variety, to the GIT quotient
	\[
		\Gr\sslash_{\!\!\mathcal{L}_{\br}} H,
	\]
	where the line bundle $\mathcal{L}_{\br}$ on the Grassmannian is linearized to correspond to the action of the algebraic torus $H = \{(t_1, \dots , t_N) \in (\C^\ast)^N : \prod t_i = 1\}$ on $\C^N$ given by identifying $\btt=(t_1, \dots , t_N) \in H$ with the diagonal unitary matrix $\operatorname{diag}(\btt^{\br} t_1, \dots , \btt^{\br} t_N)$, where $\btt^{\br} = t_1^{r_1} \cdots t_N^{r_N}$ is the character of $H$ corresponding to the vector $\br$ (cf.~\cite{foth_toric_2005,Anonymous:2004vp}).
\end{thm}

If $F \in \frames_{\blam}$ is a frame representing a point $[F] \in \Gr$, then the determinants of the $d \times d$ minors of $F$ are precisely the \emph{Pl\"ucker coordinates} of $[F]$, and in general the Pl\"ucker coordinates generate the homogeneous coordinate ring of $\Gr$. By definition, the full-spark frames are precisely those for which none of the Pl\"ucker coordinates vanish.

In turn, the homogeneous coordinate ring $\mathcal{R}$ of the GIT quotient $\Gr\sslash_{\!\!\mathcal{L}_{\br}} H$ consists of the $H$-invariant homogeneous coordinates on $\Gr$. It is known~\cite[Lemma 4.5]{Anonymous:2004vp} that $\mathcal{R}$ is spanned by monomials in the Pl\"ucker coordinates so that the total number of Pl\"ucker coordinates involving the $i$th column is $kr_i$ for some integer $k$ independent of $i$.

In our case, all $r_i = d$, so taking the product of all the Pl\"ucker coordinates and raising it to the $d$th power gives a homogeneous coordinate on $\Gr\sslash_{\!\!\mathcal{L}_{\br}} H$ whose vanishing set is exactly the collection of (equivalence classes of) spark-deficient frames in the space of (scaled) unit-norm tight frames. Since this is the vanishing set of a homogeneous coordinate, it is a subvariety, and hence its complement---the collection of full-spark frames---is open in the Zariski topology, and in particular either empty or dense. Since there are full-spark frames in each space of unit-norm tight frames (for example, the first $d$ rows of a scaled $N \times N$ discrete Fourier transform matrix~\cite{Alexeev:2012jk}), the collection of full-spark frames cannot be empty, so it must be dense.

\section{Discussion}\label{sec:discussion}

Given a $2n$-dimensional toric symplectic manifold $M$ with moment polytope $P$, one can often find \emph{action-angle coordinates} on $M$ which take the form of a map $\alpha: \operatorname{int}(P) \times U(1)^n \to M$ which inverts the momentum map $\Phi: M \to P$ in the sense that $\Phi(\alpha(p,t)) = p$. In this case, \Cref{thm:duistermaat_heckman} can be extended slightly to show that the image of $\alpha$ is a full-measure subset of $M$ and that the map $\alpha$ is measure-preserving. Sampling $P \times U(1)^n$ with respect to the product of Lebesgue measure on $P$ and the standard product measure on $U(1)^n$ and pushing forward by $\alpha$ gives a uniformly random sample from the symplectic measure on $M$ (see, for example, the discussion in~\cite{Cantarella:2016iy}).

In our setting, this means that coupling explicit action-angle coordinates on $\frames_{\blam}(\br)/(\mathrm{U}(d) \times G)$ with an algorithm for sampling $\polytope_{\blam}(\br)$ would give an algorithm for sampling random frames in $\frames_{\blam}(\br)$. In particular, by \Cref{cor:FUNTF}, such an algorithm would provide endless quantities of full spark FUNTFs.

It is natural to ask whether the analog of \Cref{thm:main} holds for real and for quaternionic frames. Symplectic geometry is not the right tool in either case, but it is very plausible that the algebraic proof sketched in \Cref{sub:algebraic approach} could be adapted to the real case to show that full spark frames are dense in real frame spaces. In a different direction, the perspective based on isotropy orbits and isoparametric submanifolds introduced in~\cite{NSquaternionicframes} seems like the most promising way to understand the measures on real and quaternionic frame spaces.

Finally, the spaces $\frames_{\blam}(\br)/(\unitary \times G)$ are examples of \emph{weight varieties}~\cite{Anonymous:1996uu}, and Goldin~\cite{Goldin:2001ba} has determined the rational cohomology ring of certain weight varieties, including the quotient $\frames_{(\lambda, \dots , \lambda)}(\br)/(\unitary \times G)$ of the space of $\lambda$-tight frames with fixed frame vector norms whenever it is a manifold. In particular, this determines the rational cohomology ring of the $(\unitary \times G)$-quotient of FUNTF space when $N$ and $d$ are relatively prime. What about in the non-manifold case or for more general frame spectra?

\subsection*{Acknowledgments}

We are very grateful for ongoing conversations about frames with various friends and colleagues, especially Jason Cantarella, Khazhgali Kozhasov, Emily King, Chris Manon, Augustin-Liviu Mare, Dustin Mixon, Louis Scharf, and Soledad Villar, and we thank the anonymous Referee \#1 for our earlier paper~\cite{NeedhamSGC}, who pushed us to explore what the symplectic machinery could say about genericity of full spark frames. This work was supported by grants from the National Science Foundation (DMS--2107808, Tom Needham; DMS--2107700, Clayton Shonkwiler) and the Simons Foundation (\#709150, Clayton Shonkwiler).

\end{document}